\documentclass[11pt, a4paper]{preprint}
\usepackage[utf8]{inputenc}
\usepackage[T1]{fontenc}
\usepackage{mathtools}
\usepackage{amsmath}
\usepackage{amsfonts}
\usepackage{amssymb}
\usepackage{amsthm}
\usepackage{enumitem}
\usepackage[normalem]{ulem}
\usepackage[mathscr]{euscript}
\usepackage{dsfont}
\usepackage{float}
\usepackage{times}
\usepackage{bbm}
\usepackage{xcolor}
\usepackage{xparse}
\usepackage{cprotect}

\usepackage{microtype}
\usepackage{mhequ}
\usepackage{xstring}

\usepackage{tcolorbox}
\usepackage{tikz}
\usepackage{tikz-cd} 

\usepackage[colorlinks,linkcolor=blue,citecolor=blue,urlcolor=blue]{hyperref}

\numberwithin{equation}{section}
\usepackage[capitalise]{cleveref}

\DeclarePairedDelimiter\abs\lvert\rvert

\newtheorem{theorem}{Theorem}[section]
\newtheorem{lemma}[theorem]{Lemma}
\newtheorem{proposition}[theorem]{Proposition}

\theoremstyle{definition}
\newtheorem{definition}[theorem]{Definition}
\newtheorem{assumption}[theorem]{Assumption}

\theoremstyle{remark}
\newtheorem{remark}[theorem]{Remark}

\newcommand{\eqdef}{\stackrel{\mbox{\tiny\rm def}}{=}}
\newcommand{\eqlaw}{\stackrel{\mbox{\tiny\rm law}}{=}}

\usetikzlibrary{calc}
\usetikzlibrary{arrows}

\def\dash{\leavevmode\unskip\kern0.18em--\penalty\exhyphenpenalty\kern0.18em}
\def\slash{\leavevmode\unskip\kern0.15em/\penalty\exhyphenpenalty\kern0.15em}

\colorlet{darkblue}{blue!90!black}
\colorlet{darkred}{red!80!black}
\colorlet{darkgreen}{green!50!black}
\let\comm\comment

\definecolor{LB}{rgb}{0.29, 0.63, 0.73}

\crefname{example}{Example}{Examples}
\Crefname{example}{Example}{Examples}

\crefname{assumption}{Assumption}{Assumptions}
\Crefname{assumption}{Assumption}{Assumptions}

\crefname{condition}{Condition}{Conditions}
\Crefname{condition}{Condition}{Conditions}

\renewcommand{\labelenumi}{(\roman{enumi})}
\renewcommand{\theenumi}{\labelenumi}

\setlist{topsep=1ex, itemsep=0.5ex, before={\setlist{topsep=-.5ex}}}

\DeclareMathOperator{\supp}{supp}

\newcommand{\B}{\ensuremath{\mathcal{B}}}
\newcommand{\C}{\ensuremath{\mathcal{C}}}
\newcommand{\D}{\ensuremath{\mathcal{D}}}

\newcommand{\F}{\ensuremath{\mathcal{F}}}

\newcommand{\N}{\ensuremath{\mathbb{N}}}

\newcommand{\R}{\ensuremath{\mathbb{R}}}

\newcommand{\1}{\ensuremath{\mathbf{1}}
}

\def\ff#1#2{\textstyle{\frac{#1}{#2}}}
\let\f\frac
\def\epsilon{\varepsilon}
\def\le{\leq}

\def\E{\mathbb E}

\newcommand{\norm}[1]{\ensuremath\left\|#1\right\|}

\renewcommand{\geq}{\geqslant}
\renewcommand{\leq}{\leqslant}
\def\CC{{\mathcal C}}

\def\${|\!|\!|}
\def\B{{\mathcal B}}
\def\F{{\mathcal F}}
\def\<{\langle}
\def\>{\rangle}
\def\CT{\mathcal{T}}
\def\CR{\mathcal{R}}
\def\dom{\mathrm{Dom}}
\def\cov{\mathrm {cov}}
\newcommand{\vertiii}[1]{{\left\vert\kern-0.25ex\left\vert\kern-0.25ex\left\vert #1 \right\vert\kern-0.25ex\right\vert\kern-0.25ex\right\vert}}
\newcommand{\rom}[1]{(\textup{\uppercase\expandafter{\romannumeral#1}})}

\definecolor{LB}{rgb}{0.29, 0.63, 0.73}
\definecolor{dg}{RGB}{85,168,104}

\newcommand{\Cc}{\mathcal{C}_c} 
\newcommand{\Ccinf}{\mathcal{C}_c^\infty} 
\newcommand{\cH}{\mathcal{H}}
\newcommand{\bD}{\mathbb{D}} 
\newcommand{\cS}{\mathcal{S}} 

\newcommand{\cU}{\mathcal{U}} 
\newcommand{\fr}{\Phi}
\newcommand{\gfr}{\Psi}

\def\scal#1{\langle #1\rangle}

\newcommand{\cI}{\mathcal{I}} 

\newcommand{\rhosy}{\rho_{s,y}} 
\newcommand{\chiB}{\underline{\chi}^\epsilon_{\tau}}
\newcommand{\chiU}{\overline{\chi}^\epsilon_{\tau}}

\def\mfK{\mathfrak{K}}

\let\eps\varepsilon
\let\d\partial

\makeatletter

\DeclareRobustCommand{\TitleEquation}[2]{\texorpdfstring{\StrLeft{\f@series}{1}[\@firstchar]$\if%
b\@firstchar\boldsymbol{#1}\else#1\fi$}{#2}}

\DeclareRobustCommand{\cev}[1]{%
  {\mathpalette\do@cev{#1}}%
}
\newcommand{\do@cev}[2]{%
  \vbox{\offinterlineskip
    \sbox\z@{$\m@th#1 x$}%
    \ialign{##\cr
      \hidewidth\reflectbox{$\m@th#1\vec{}\mkern7mu$}\hidewidth\cr
      \noalign{\kern-\ht\z@}
      $\m@th#1#2$\cr
    }%
  }%
}

\makeatother
\usepackage[colorlinks,linkcolor=blue,citecolor=blue,urlcolor=blue]{hyperref}

\begin{document}
\title{Scaling limit of the KPZ equation with non-integrable spatial correlations
}
\author{
	Luca Gerolla$^1$, 
    Martin Hairer$^{1,2}$,  Xue-Mei Li$^{1,2}$
}

\institute{
Imperial College London, SW7 2AZ London, UK.
\and
EPFL, 1015 Lausanne, Switzerland.\\
\email{luca.gerolla16@imperial.ac.uk, martin.hairer@epfl.ch, xue-mei.li@epfl.ch}}

\maketitle
\begin{abstract}
We study the large scale fluctuations of the KPZ equation in dimensions $d \geq 3$ driven by Gaussian noise that is 
white in time Gaussian but features non-integrable spatial correlation with decay rate $\kappa \in (2, d)$ and a suitable limiting profile. We show that its scaling limit is described by the corresponding additive stochastic heat equation. 
In contrast to the case of compactly supported covariance,  the noise in the stochastic heat equation retains spatial correlation with covariance $|x|^{-\kappa}$. 
Surprisingly, the noise driving the limiting equation turns out to be the scaling limit 
of the noise driving the KPZ equation so that, under a suitable coupling, one has convergence in probability,
unlike in the case of integrable correlations where fluctuations are 
enhanced in the limit and convergence is necessarily weak.

\vspace{1em}


\end{abstract}

\setcounter{tocdepth}{2}
\tableofcontents
\pagestyle{plain}

\section{Introduction}
\label{sec:intro}

In this article, we investigate the large-scale dynamics of  the Kardar--Parisi--Zhang (KPZ) equation on $\R^d$ in dimension $d \geq 3$. The KPZ equation 
\begin{equ}\label{eq:KPZ}
	\partial_t h = \tfrac{1}{2} \Delta h +
	\tfrac{1}{2} |\nabla h |^2 
	+ \beta \xi\;, 
\end{equ}
was originally introduced in \cite{kardar1986dynamic} to model random interface growth. Here, $\beta$ is a coupling constant, and $\xi$ is a mean zero Gaussian noise. Over time, it has evolved into one of the most extensively studied stochastic PDEs due to its ability to capture the universal behaviour of many probabilistic models on one hand and the mathematical challenges it presents for a well-posed solution theory on the other.
Very recently, significant achievements have been made, demonstrating that, in dimension $1$, the rescaled solution 
$\epsilon^{\frac 12} h({\epsilon^{-\frac 32}}t,  \epsilon^{-1} x ) - C_\eps t$ converges to the KPZ fixed point \cite{quastel2020convergence, virag2020heat}. This accomplishment built upon the remarkable progress made in the last couple of decades \cite{Bertini-Giacomin-97, baik1999distribution, amir_corwin_quastel_2011probability, Balazs-Quastel-Seppalainen2011,alberts_khanin_quastel2014intermediate, matetski_quastel_remenik2016-21kpz}.
While the KPZ equation offers insights into universal phenomena, it poses substantial mathematical challenges when it comes to defining a robust and mathematically rigorous solution that exhibits the aforementioned physical relevance
 \cite{hairer_solvingKPZ, hairer_RegStr, gubinelli2015paracontrolled,gubinelli2017kpz,kupiainen2016renormalization}.

In the original model, the mean zero Gaussian noise  $\xi$ has short range correlations.
In that case, it was shown recently in a number of works 
\cite{ magnen2018scaling, gu18edwards, cosco2019spacetime_fluct, dunlap2020_KPZ_fluctuations, cosco2020law, lygkonis2020edwards} that the large-scale behaviour of the KPZ equation in dimensions $d \geq 3$ in the weak coupling regime
is described by the Edwards--Wilkinson model, namely simply the additive stochastic heat equation driven by space-time white noise. 

In the present article, we consider instead spatially coloured noise with long range correlations, and we are interested to explore how these correlations affect the KPZ scaling limit in the supercritical regime $d \geq 3$. 
In such settings, where the correlation is not integrable, so that the form of the effective variance in the above references is infinite, an effective fluctuation theory for the stochastic heat equation (SHE) has been developed \cite{GHL23_arxiv_preprint}. Like there, we consider large scales correlations of polynomial decay $|x|^{-\kappa}$. Our result on the KPZ scaling limit introduced below aligns with  
 the physics literature \cite{KatzavSchwartz99, PhyKPZ14} which predicts a Gaussian limit when $\kappa >2$. 

We begin by describing the noise $\xi$ considered:
\begin{assumption}
\label{assump-cov-KPZ} 
The Gaussian noise $\xi$ is white in time with smooth spatial covariance $R$,  formally
$\E[ \xi(t,x)\xi(s,y)] = \delta(t-s) R(x-y)$. There exists $\kappa \in (2,d)$ such that
$R(x) \lesssim (1 + |x|)^{-\kappa}$
and, for any $x \in \R^d \setminus \{0\}$, 
$\lim_{\eps \to 0}\epsilon^{-\kappa} R( \epsilon^{-1}x ) = |x|^{-\kappa}$.
\end{assumption}
For example, the noise obtained by convolving (in space) a space-time white noise with a smooth function $\phi :\R^d\to \R_+$ behaving at infinity like $\phi(x) \sim |x|^{-\frac{d+\kappa} {2}}$ leads to $R(x) = \int_{\R^d} \phi(x+y) \phi(y) dy$, which satisfies Assumption~\ref{assump-cov-KPZ}.

Writing $h$ for the solution to \eqref{eq:KPZ}, but with a slowly varying initial condition of the form
$\eps^{\f\kappa2-1}h_0(\eps x)$, we rescale it by setting, for $\eps \in (0,1]$,
\begin{equ}
	h_\eps(t,x) = \eps^{1-\f\kappa2}h(t/\eps^2,x/\eps)-\tfrac t2 \beta^2 R(0) \epsilon^{-1-\f{\kappa}{2}}\;.
\end{equ}
A simple calculation shows that $h_\eps$ then solves the rescaled KPZ equation
\begin{equ}[eq:KPZ_rescaled]
\partial_t h_\eps = \tfrac{1}{2} \Delta h_\eps +
	\tfrac{1}{2} \eps^{\f\kappa2-1} |\nabla h_\eps |^2 - \tfrac12 \beta^2 R(0) \epsilon^{-1-\f{\kappa}{2}}
	+ \beta \xi^\eps\;,
\end{equ}
with initial condition $h_0$.
Here $\xi^\epsilon$ is the rescaling of $\xi$ such that  
\begin{equ} \label{eq:xi_eps_cov}
	\E [ \xi^\epsilon (t,x) \xi^\epsilon (s,y) ] = \delta(t-s) \eps^{-\kappa} R(\tfrac {x-y} \epsilon)\;.
\end{equ}
Under this scaling, $\xi^\eps$ converges in law,  as $\eps \to 0$,   to a non-trivial Gaussian limit $\xi^0$
 with covariance
\begin{equ} \label{eq:xi0_covariance}
	\E [\xi^0(t,x) \xi^0(s,y)]= \delta(t-s) |x-y|^{-\kappa}\;.
\end{equ}
It will be convenient to upgrade this convergence to convergence in probability (which can always be achieved
by Skorokhod's representation theorem):

\begin{definition}
A coupling between the $\xi^\eps$ is \textit{good} if, for every $\psi \in \CC_c^\infty(\R^{d+1})$, 
there exists a random variable $\xi^0(\psi)$ such that $\xi^\eps(\psi) \Rightarrow \xi^0(\psi)$ in probability.
\end{definition}

If one can write
$\xi(t,x) = \int_{\R^d} \phi(x-y)\eta(t,y)\,dy$,
for $\eta$ a space-time white noise and for $\phi$ as described just above, 
then a natural good coupling is given by
\begin{equ}\label{good-coupling}
\xi^\eps(t,x) = \eps^{-(d+\kappa)/2} \int_{\R^d} \phi((x-y)/\eps)\eta(t,y)\,dy\;.
\end{equ}

Given the limiting noise $\xi^0$, we write
 $\cU$ for the solution to the corresponding additive stochastic heat equation:
	\begin{equ}[eq:AddSHE_h0]
		\partial_t \cU =\ff 12 \Delta \cU  + \beta \xi^0\;, \qquad \cU(0, \cdot) = h_0\;.
	\end{equ}
Our main result is then the following.

\begin{theorem} \label{thm:main-theorem-KPZ}
Let $d\geq 3$ and let $\xi$ satisfy Assumption~\ref{assump-cov-KPZ}. 
For any $p\geq 1$, $\alpha <  1 - \f \kappa 2$, and $\sigma <  - 1 - \f \kappa 2$,  there exists a constant $c<0$ and a  positive value $\beta_0 = \beta_0(p, \alpha) $  such that for any $\beta<\beta_0$, for any deterministic
initial condition $h_0 \in \CC_b(\R^d)$, and any good coupling of the noise,
one has, for any $T > T_0 > 0$ 
\begin{equ}
	\lim_{\eps \to 0} \bigl(h_\eps(t,x) - \eps^{1-\f\kappa2}c\bigr) = \cU\;, 
\end{equ}
in probability in $L^p([T_0, T], \C^{\alpha}(E))$, for any compact $E \subset \R^d$.
\end{theorem}

We provide the definition and comment on the constant $c$ in Remark~\ref{rem:constant_c} below, 
whilst the proof of this result will be given in Section~\ref{sec:proofMain}, where we also provide the
definition of the spaces $\C^{\alpha}(E)$ and their semi-norms.
In fact, we will show a limit theorem for general functions $\Phi(u)$ of  solutions to the 
multiplicative stochastic heat equation
\begin{equ}
	\label{eq:basic_linear_mSHE}
	\partial_t u(t,x)=\tfrac 12 \Delta u (t,x)+\beta u(t,x) \xi(t,x)\;, \qquad u(0,\cdot) = 1\;,
\end{equ}
from which Theorem~\ref{thm:main-theorem-KPZ} essentially follows as a special case with $\Phi = \log$.
(We still need a little bit more work since we allow for general initial conditions in 
Theorem~\ref{thm:main-theorem-KPZ}, while our general fluctuation result, Theorem \ref{thm:gen_transf-convergence-light}, holds for $u_0 \equiv 1$.
It is this additional work that is performed at the end of Section~\ref{sec:proofMain}.)
Under the hypotheses stated earlier, \eqref{eq:basic_linear_mSHE} admits a unique 
solution \cite[Thm~0.2]{peszat2000nonlinear}.
Since we are interested in its large-scale fluctuations, we write 
$u_\epsilon(t,x) = u(t/\epsilon^2, x/\epsilon)$
for the diffusively rescaled solution,
which solves
		\begin{equ}\label{eq:rescaled_mSHE}
		\partial_t u_\epsilon(t,x)=\ff 12 \Delta u_\epsilon(t,x)+\beta \epsilon^{\frac \kappa 2 -1}  u_\epsilon(t,x) \xi^\epsilon(t, x)\;, \qquad u_\epsilon(0, \cdot)=1,
	\end{equ}
with $\xi^\epsilon$ as before.
We then have the following functional central limit theorem, which will be restated as 
Theorem~\ref{thm:gen_transf-convergence} below and proven in Section~\ref{sec:proofMain}.

\begin{theorem}\label{thm:gen_transf-convergence-light}  
Consider a map $\fr : \R_+ \to \R$ satisfying Assumption \ref{transformation} and let
\begin{equ}
	u_\eps^\Phi(t,x) = 
	\eps^{1-\frac \kappa 2} \bigl( \fr(u_\eps(t,x))-\E [\fr(u_\eps(t,x))]\bigr)\;.
\end{equ}
With assumptions on $d$, $R$, $p$, $\alpha$ and $\sigma$ as in Theorem~\ref{thm:main-theorem-KPZ},   there exists $\beta_0^{\fr}=\beta_0^{\fr}(\alpha, p) >0$ and $\nu_{\fr} \geq 0$ such that, for all $\beta < \beta_0^{\fr}$, for any $T >0$,  and any good coupling, $u_\eps^\Phi$
converges in probability to $ \nu_{\fr} \cU^0$ in the space $L^p([0, T], \C^{\alpha}(E))$ for any compact $E \subset \R^d$. Here, $\cU^0$ denotes the solution to the additive stochastic heat equation
	\begin{equ}\label{eq:EW_eqn}
		\partial_t \cU^0 =\ff 12 \Delta \cU^0 + \beta \xi^0\;, \qquad \cU^0(0, \cdot) = 0\;.
	\end{equ}
\end{theorem}

Our class of possible transformations includes the Cole--Hopf transform $\Phi(u) = \log u$, for 
which one has $\nu_\fr = 1$. In this case we recover the fluctuations of KPZ since
$h_\eps (t,x) = \eps^{1-\f\kappa2} \log u_\epsilon(t,x)$ solves \eqref{eq:KPZ_rescaled} with flat initial conditions $h_0 \equiv 0$. 
We discuss and provide the explicit expression of $\nu_{\fr}$ in the general case
in Remark~\ref{rem:constant_eff_var} below. 

The main tool  (and this is also
one of the main novelties of this article) to prove the above is a homogenisation result for the fundamental solution of \eqref{eq:basic_linear_mSHE}.
In order to formulate it, write $w_{s,y}(t,x)$ for the solution to 
\eqref{eq:basic_linear_mSHE} at times $t \geq s$ with initial condition at time $s$ given by
$w_{s,y}(s,\cdot) = \delta_y$:

\begin{theorem}\label{theo:homogenisation}
Let $d\geq 3$ and $\xi$ satisfies Assumption~\ref{assump-cov-KPZ}. Then for any $p\geq 1$, for sufficiently small $\beta$,
 there exist stationary random 
fields $\vec Z$ and $\cev Z$ such that, with $\mu = \f 12 -\f 1\kappa$, 
\begin{equ}[e:mainHomog]
\Big\|\frac{w_{s,y}(t,x)}{P_{t-s}(y-x)} - \vec Z(t,x)\cev Z(s,y)\Big\|_p
\lesssim 
(1 + t-s)^{-\mu} (1 + (1 + t-s)^{-\f 1 2}|x-y|)
\;.
\end{equ}
\end{theorem}

A slightly more detailed version of this statement will be formulated as 
Proposition~\ref{prop:rho_homogenisation} below.
The field $\vec Z$ is the analogue of the pullback field already appearing in
\cite{GHL23_arxiv_preprint,dunlap2018homo, cosco2019spacetime_fluct, cosco2020renormalizing} where it was denoted by $Z$. It
is obtained as the limit $\vec Z(t,x) = \lim_{s \to -\infty} u^{(s)}(t,x)$, where
$u^{(s)}$ denotes the solution to \eqref{eq:basic_linear_mSHE} with initial condition $1$ at time $s$. It turns out that the field $\cev Z$ is given by ``$\vec Z$ running backwards in time'' 
(which explains our notations) in a sense that will be made precise below, see Theorem \ref{thm:stat_field} and equation \eqref{cev-Z}.
A shadow of the field $\cev Z$ can be
seen in the constant $\bar c$ appearing in \cite[Eq.~1.6]{dunlap2018homo}, whose inverse would 
correspond to the expectation of $\cev Z$ (which in our case equals $1$ as a consequence of the
fact that we consider noise that is white in time). 
More explicitly, these fields satisfy the equations
\begin{equs}
		\vec Z(t,x)&=1 + \beta \int_{-\infty}^t\int_{\R^d} P_{t-r}(x-z) \vec Z(r,z)\,\xi(dr,dz)\;,\label{e:vecZ}\\
		\cev Z(s,y)&=1+\beta \int_s^\infty \int_{\R^d} P_{r-s}(y-z)\cev Z(r,z)\,\xi(dr,dz)\;,\label{e:cevZ}
\end{equs}
with the stochastic integral appearing in the second expression being of backwards Itô type.

Using the Clark--Ocone formula we recover a divergence representation for the fluctuations of $\fr (u)$, which can be expressed in terms of the fundamental solution of \eqref{eq:basic_linear_mSHE}. Relying on the large time homogenisation of $w_{s, y}$ from Theorem~\ref{theo:homogenisation} and mixing of the stationary fields, we are able to approximate the fluctuations of $\fr(u)$ by the fluctuations of the linear SHE  (Proposition~\ref{prop:fluctuations_approx}) and then prove Theorem~\ref{thm:gen_transf-convergence-light}. 
We shall explain more in details the argument and carry out the proof in Section~\ref{sec:main}.


\begin{remark}\label{rem:constant_c}
Unlike what happens in the case of short-range correlations, the constant $c$ in Theorem~\ref{thm:main-theorem-KPZ} is all that is left from the KPZ nonlinearity since neither the covariance of the noise nor the coefficient in front of the Laplacian end up being affected by the
scaling limit.
The constant is given by $c=\E[\log Z]$ with $Z \eqlaw \vec Z(0,0)$. Since $\E Z = 1$ and $Z$ has non-zero variance
by \eqref{e:vecZ}, $c$ is strictly negative.
The way one can intuitively understand the presence of $c$ is the following: since we start with a fixed $\eps$-independent bounded continuous initial condition $h_0$ which looks quite different from
the stationary solutions (which exhibit oscillations of size about $\eps^{1-\f\kappa2}$), 
the KPZ nonlinearity
and the renormalisation constant appearing in \eqref{eq:KPZ_rescaled} don't quite balance 
out initially. By the time the process
has become locally sufficiently close to stationarity, 
this imbalance has already caused it to shoot
down by a substantial amount.

This strongly suggests that one can find a collection $Y_\eps$ of 
stationary processes such that $\lim_{\eps \to 0} Y_\eps = 0$ in distribution
and such that
a statement analogous to Theorem~\ref{thm:main-theorem-KPZ}, but without 
the constant $c$, holds provided that one starts $h_\eps$  with
initial condition $Y_\eps + h_0$. In fact, we believe that one should be able to take
\begin{equ}
Y_\eps(x) \eqlaw \eps^{1-\f\kappa2}\bigl(\log (u_\eps(t_\eps,x)) - c\bigr)\;,
\end{equ}
(but independent of the noise $\xi$) for a suitable sequence of times $t_\eps \to 0$, and show that $\lim_{\epsilon\to 0}h_\epsilon=\cU$.
\end{remark}

\begin{remark}\label{rem:constant_eff_var}
	With the notations above at hand, one has $\nu_{\fr} = \E [Z \,\fr' (Z)]$ in Theorem~\ref{thm:gen_transf-convergence-light} where $Z=\vec Z(0,0)$. 
	We will see in Lemma~\ref{lem:transform_moments+decorrelation} 
	that our assumptions on $\fr$
	guarantee that this quantity is finite.
	Given the dependence of the pullback stationary solution $\vec Z$ on $\beta$, we note that  $\nu_{\fr}$ actually depends on $\beta$ in general and that, as $\beta\to 0$, $\nu_{\fr}$ converges to $\Phi'(1)$.
	In the special case $\fr(x)=x$, one has $\nu_\fr = 1$ and we recover the result of  \cite{GHL23_arxiv_preprint},
	at least in the case $\sigma(u) = u$.
	
	Note that the effective
	variance differs from that obtained for the case of integrable $R$ in 
	\cite{gu18edwards, magnen2018scaling, dunlap2020_KPZ_fluctuations,cosco2020law, lygkonis2020edwards, gu20_nlmSHE, dunlap2018homo, cosco2019spacetime_fluct}. There, the effective variance 
	$\nu_{\textrm{eff}}$ is given by the expression
	\begin{equ}
		\nu^2_{\textrm{eff}}  = \nu_\Phi^2 \int_{ \R^d} R(x) \E \big[ \vec Z(0,x) \vec Z(0,0)\big] dx\;.
	\end{equ}
	In their works, the fluctuation scale depends on the dimension, while  in our setting the fluctuations scale $\epsilon^{\kappa -2}$ depends on the decay coefficient $\kappa$ of the noise covariance.
\end{remark}

\subsubsection*{Structure of the article}

The remainder of the article is structured as follows. 
In Section~\ref{sec:stationary}, we introduce the random fields $\vec Z$ and $\cev Z$ and we collect some useful a priori bounds on these fields as well as on the solutions to the multiplicative stochastic heat equation \eqref{eq:basic_linear_mSHE} and their Malliavin derivatives.
Section~\ref{sec:Homogenisation} forms the bulk of the article and contains the proof of the homogenisation
result. Finally, in Section~\ref{sec:main}, we use all of these ingredients to provide a proof of Theorems~\ref{thm:main-theorem-KPZ} and~\ref{thm:gen_transf-convergence-light}.

\subsubsection*{Acknowledgements}

The authors are grateful to Alex Dunlap and Yu Gu for discussions on this topic and in particular
for pointing out that our proof may yield convergence in probability in this setting.
This research was partially supported by the Royal Society through MH's research professorship.  
XML acknowledges
support from EPSRC grant EP/V026100/1 and EP/S023925/1. LG is supported by the EPSRC Centre for Doctoral Training in Mathematics of 
Random Systems: Analysis, Modelling 
and Simulation (EP/S023925/1).

\section{Analysis of the stationary solution}
\label{sec:stationary}

\subsection{Notations and conventions}\label{sec:notation}
The background probability space is  $(\Omega, \F, \mathbb{P})$, and we write $\F_t $ for the filtration 
generated by $\xi$, 
i.e.\ $\F_t$ 
is the smallest $\mathbb{P}$-complete $\sigma$-algebra such that the random variables 
$\{\xi( \phi) \,:\, \supp\phi \subset (-\infty,t]\times \R^d, \phi\in \C_c^\infty (\R_+\times \R^d)\}$ are all $\F_t$-measurable.  We generally interpret an integrable function $\xi$  as a distribution by the following expression:
$$\xi(\phi)=\int_{\R\times \R^d} \xi(s,y)\phi(s,y) dsdy.$$
The notation $\| \cdot \|_p$ stands for the $L^p(\Omega)$ norm. Moreover, 
\begin{itemize}
	\item   With $a\lesssim b$ we mean $a \leq C b$ for some constant $C>0$, independent of $\epsilon$. 
\item
 $\Ccinf(\R^d)$ denotes the set of  smooth functions on $\R^d$ with compact support. 
 \item For $E\subset \R^d$ open and $k\in \N$, we denote with  $\C^k(E)$ the space of $k$ times continuously differentiable functions on $E$ and $\C_b^k(E)$ the subset of bounded $\C^k(E)$ functions with all partial derivatives up to order $k$ bounded.
\end{itemize}

\subsection{Forward and backward stationary solutions}

For the purpose of this section, it will be convenient to assume that the underlying
probability space $\Omega$ is the space of space-time distributions and that the noise $\xi$ is simply
the canonical process.

With $u^{(s)}(t,x)$ denoting the solution to \eqref{eq:basic_linear_mSHE} 
with initial condition $u^{(s)}(s,\cdot) \equiv 1$ at time $s$, we recall \cite[Thm~2.10]{GHL23_arxiv_preprint}. To avoid confusion, we remark that in \cite{GHL23_arxiv_preprint}, $u^{(K)}(t,x)$ denotes the solution that starts from time $-K$.   A result of similar nature can be found in \cite[Thm 1.1]{gu20_nlmSHE}.

\begin{theorem}\label{thm:stat_field}
Let $d\geq 3$ and $\xi$ satisfies Assumption~\ref{assump-cov-KPZ}. Given any $p\geq1$, there exists a positive constant 
$\beta_1(p) = \beta_1(p, d, R)$ and a space-time stationary random field 
$\vec Z$ such that, for $\beta \leq \beta_1(p)$, one has 
\begin{equation}\label{eq:Lp_cvg_rate}
	\big\|u^{(s)}(t, x) - \vec Z(t,x)\big\|_{L_p(\Omega)}\lesssim 1 \wedge (t-s)^{\frac {2-\kappa} 4}\;,
\end{equation} 
uniformly over all $s \le t$ and $x \in \R^d$. Recall that $\kappa\in (2,d)$. Furthermore, for
any compact region $\mfK \subset \R^{d}$ and every $t \in \R$, one has
$$\lim_{s \to -\infty} \E \sup_{x \in \mfK} |  u^{(s)}(t,x)-\vec {Z}(t,x) |^p \to 0.$$
\end{theorem}

It follows from \eqref{eq:Lp_cvg_rate} that $\vec Z$ is a random fixed point for the random
dynamical system generated by the stochastic heat equation \eqref{eq:basic_linear_mSHE}.
In fact, we have even more structure as we will describe now.
Our probability space admits  a natural ``time reversal'' involution $\CR \colon \Omega \to \Omega$  and 
``space-time translations'' $\CT_{t,x} \colon \Omega \to \Omega$ such that, formally
\begin{equ}
\bigl(\CR \xi\bigr)(t,x) = \xi(-t,x)\;, \qquad
\bigl(\CT_{t,x} \xi\bigr)(s,y) = \xi(s+t,x+y)\;,
\end{equ}
with the obvious rigorous interpretations in terms of test functions. 
These operations naturally yield linear operators acting on the space of random variables 
$A \colon \Omega \to \R$ by 
\begin{equ}[e:opRandom]
(\CR A)(\xi) \eqdef A(\CR\xi)\;,\qquad
(\CT_{t,x} A)(\xi) \eqdef A(\CT_{t,x}\xi)\;.
\end{equ}
Since the law of $\xi$ is invariant under $\CR$ and $\CT_{t,x}$, the operations in
\eqref{e:opRandom} are isometries for every $L^p(\Omega,\mathbb{P})$-space.

We will repeatedly use the fact that, by uniqueness of solutions to \eqref{eq:basic_linear_mSHE}, 
one has the following.
\begin{lemma}\label{lem:RDS}  
For any $t\in \R$, $x\in \R^d$, one has $\CT_{t,x} \CR  = \CR \CT_{-t,x} $ and the following holds almost surely:
\pushQED{\qed}
\begin{align*}
u^{(s)}(t,x)(\CT_{0,y}\xi)&=u^{(s)}(t, x+y)(\xi) \;,\\ 
u^{(s)}(t,x)(\CT_{r,0}\xi)&=u^{(s+r)}(t+r, x)(\xi)\;. \qedhere
\end{align*}
\popQED
\end{lemma}

Given any random variable $A$ as above,
we then define the stationary random fields
\begin{equ}\label{stationary-field}
\vec A(t,x) = \CT_{t,x} A\;,\quad
\cev A(t,x) = \CT_{t,x}\CR A\;.
\end{equ}
If we now define
\begin{equ}[e:defZs]
Z = \lim_{s\to -\infty} Z_s\;,\qquad Z_s = u^{(s)}(0,0)\;,
\end{equ}
then we see that this
definition is consistent with our notation since by Lemma~\ref{lem:RDS} $\vec Z(t,x):=\CT_{t,x} Z$,   defined as 
in \eqref{stationary-field}, satisfies
\begin{equs}\label{vec-Z}
\vec Z(t,x) &= \lim_{s\to-\infty} u^{(s)}(0,0)(\CT_{t,x}\xi)
= \lim_{s\to-\infty} u^{(t+s)}(t,x)(\xi)\\
&= \lim_{s\to -\infty} u^{(s)}(t,x)\;,
\end{equs}
which is indeed the same quantity appearing in (\ref{eq:Lp_cvg_rate}), as well as in 
\eqref{e:mainHomog}, by virtue of \eqref{e:vecZ}.
With the notation above
\begin{equ}\label{cev-Z}
\cev Z(s,y)=\CT_{s,y} \CR Z = \CR \CT_{-s,y} Z= \CR \, \vec Z(-s, y)\;,
\end{equ}
so that
\begin{equ}
\cev Z(s,y)(\xi)=\lim_{r\to -\infty}u^{(r)}(-s,y)(\CR \xi)\;,
\end{equ}
which is indeed consistent with \eqref{e:cevZ} via the change of variables $r \mapsto -r$
(and noting that such a change of variables does indeed turn Itô integrals into backwards Itô integrals by
simple Riemann sum approximations).

Let us now explain why $\cev Z$ appears in \eqref{e:mainHomog}. If it were the case that 
\begin{equ}
w_{s,y}(t,x) = P_{t-s}(y-x) \vec Z(t,x)\cev Z(s,y)\;,
\end{equ}
then, as a consequence of the fact that $\E \vec Z = 1$ and the law of large number 
we could recover $\cev Z$ by 
\begin{equ}
\cev Z(s,y) = \lim_{t \to \infty}  \int_{\R^d} w_{s,y}(t,x)\,dx\;.
\end{equ}
It is therefore natural to consider the stochastic process
\begin{equ}[e:defn_Csy]
C_{s,y}(t) = \int_{\R^d} w_{s,y}(t,x)\,dx\;.
\end{equ}
As we will see shortly, it is then indeed the case that one has the identity 
$\cev Z(s,y) = \lim_{t \to \infty} C_{s,y}(t)$.
In fact we will show more, namely that with $Z_s$ as in \eqref{e:defZs}, one has the exact identity
$C_{s,y}(t) = \cev Z_{s-t}(s,y)$.


\begin{proposition} \label{prop:C_Z_relation}
Under the conditions of \Cref{thm:stat_field},
one has $C_{s,y}(t) =\cev Z_{s-t}(s,y)$, 
and
$$\|C_{s,y}(t)-\cev Z(s,y)\|_{p}\lesssim 1 \wedge (t-s)^{\frac {2-\kappa} 4}.$$
\end{proposition} 

\begin{proof}
By the definition (\ref{stationary-field}), $\cev Z_{s-t}(s,y):=\CT_{s,y} \CR  Z_{s-t}=\CR \CT_{-s,y}  u^{(s-t)}(0,0)$.
By Lemma~\ref {lem:RDS},
\begin{equation}\label{cev-Z}
\cev Z_{s-t}(s,y)=\CR u^{(-t)}(-s,y), \qquad \lim_{t\to \infty} \cev{Z}_{s-t}(s,y)=\cev Z(s,y).
\end{equation}

We shall use the method of duality. The basic observation is that if $u$ denotes the SHE running forward in time
with initial condition  $u_s$ at some time $s$
and $v$ denotes the one running backwards in time with terminal condition $v_t$ at some time 
$t > s$ then $\scal{u_r, v_r}$, where the bracket denotes the inner product on $L^2(\R^d)$, is independent of 
$r \in [s,t]$. Formally, this is because
\begin{equ}[e:formal]
\d_r \scal{u_r, v_r}
= \scal{\Delta u_r + \xi u_r, v_r} - \scal{u_r, \Delta v_r + \xi v_r} = 0\;.
\end{equ}


One can make this precise in the following way. 
Fix a mollifier $\rho \in \CC_c^\infty(\R \times \R^d)$ such that $\int \rho = 1$ and
$\rho(t,x) = \rho(-t,x) = \rho(t,-x)$. For any $\delta > 0$, we write $\rho_\delta(t,x) = \delta^{-2-d}\rho(t/\delta^2,x/\delta)$ and we set $\xi_\delta(t,x) = (\CT_{t,x}\xi)(\rho_\delta)$.

Fix then some initial time $s$ and constant $K_\delta$ and consider the solution $u^\delta(r,x)$ to the random PDE
\begin{equ}[e:defudelta]
\d_r u^\delta = \ff12 \Delta u^\delta + \beta u^\delta \xi_\delta - K_\delta u^\delta\;,\quad
u^\delta(s,\cdot) = \delta_y\;.
\end{equ}
We also fix some final time $t$ and consider the solution $v^\delta$ to the random \textit{backwards} PDE 
\begin{equ}[e:defvdelta]
\d_r v^\delta = -\ff12 \Delta v^\delta - \beta v^\delta \xi_\delta + K_\delta v^\delta\;,\quad
v^\delta(t,\cdot) = 1\;.
\end{equ}
At this stage we note that if we set $\tilde v^\delta(r,x) = \CR v^\delta(-r,x)$, then 
$\tilde v^\delta$ satisfies the \textit{forwards} PDE
\begin{equ}[e:defvtilde]
\d_r\tilde v^\delta = \ff12 \Delta \tilde v^\delta + \beta \tilde v^\delta \xi_\delta - K_\delta \tilde v^\delta\;,\quad
\tilde v^\delta(-t,\cdot) = 1\;.
\end{equ}
At this point we note that \eqref{e:defudelta} fits into the setting of \cite[Thm~5.2]{Cyril}
(in fact, the regularity of our limiting noise $\xi$ is better than what is considered there so some
of the arguments could be simplified). Combining this with \cite{Rhys} which shows that the BPHZ
lift of $\xi^\delta$ to the regularity structure of \cite[Sec.~2.1]{Cyril} converges as $\delta \to 0$,
we conclude that, provided that the constants $K_\delta$ are chosen appropriately, $u^\delta$
converges as $\delta \to 0$ to the Green's function $w_{s,y}$. This convergence takes place in
an exponentially weighted space that enforces rapid decay of the solution. 

By the exact same argument, one finds that $\tilde v^\delta$ converges to $u^{(-t)}$ as 
$\delta \to 0$, this time in a weighted space that allows for some slow growth of the solutions. 
In particular, this shows that $v^\delta_r$ converges to 
\begin{equ}
v_r(x) = \CR u^{(-t)}(-r,x) 
 = \cev Z_{r-t}(r, x)\;. 
\end{equ}
We conclude that, for any $r \in [s,t]$, one has
\begin{equ}
\lim_{\delta \to 0} \scal{u_r^\delta,v_r^\delta} = \int_{\R^d} w_{s,y}(r,x) \cev Z_{r-t}(r, x)\,dx\;.
\end{equ}
On the other hand, it follows immediately from \eqref{e:defudelta} and \eqref{e:defvdelta} that
$\scal{u_r^\delta,v_r^\delta}$ is independent of $r$, as in the formal calculation \eqref{e:formal}.
Since $Z_0 = 1$ and $\cev Z_0(t, \cdot)=1$, it follows that 
\begin{equ}
\cev Z_{s-t}(s, y)
= \int_{\R^d} w_{s,y}(t,x) \cev Z_{0}(t, x)\,dx
= C_{s,y}(t)\;,
\end{equ}
which is precisely as claimed. The rate of convergence follows from (\ref{cev-Z}) and \Cref{thm:stat_field}.
\end{proof}

\subsection{A priori bounds}

Our main goal in this subsection is to establish $L^p$ bounds on the Malliavin derivative of the random variable $u$
as well as on $\Phi(u)$. We write $\cH$ for the reproducing kernel space of the noise $\xi$, 
which is formed by completing the space $\Ccinf(\R \times \R^d)$ under the inner product 
\begin{equation}\label{Hilbert}
	\langle f , g  \rangle_\cH = \int_{-\infty}^{\infty} \int_{\R^{2d}} f(s, y_1) g(s, y_2) R(y_1-y_2) \, dy_1\,  dy_2\,  ds\;, 
\end{equation}
after quotienting out zero norm elements.
The family of random variables $$\xi(f)=\int f(s,y) \, \xi(ds, dy)\;,\qquad f\in \cH\;,$$
defines an isonormal Gaussian process on $\cH$: $\E [  \xi(f) \xi(g)] =  \langle f , g  \rangle_\cH $. 
Denote by $D$ the Malliavin derivative $D$ of a random variable, $D(\xi(f))=f$. Let $ \C_p^\infty(\R^n)$ denote the space of smooth functions $F \in  \C^\infty(\R^n)$ with $F$ and its partial derivatives growing at most polynomially. 
Denoting by $\cS$ the class of smooth random variables of the form $F(\xi(f_1), \dots, \xi(f_n))$, we set 
\begin{equ}
	D F(\xi(f_1), \dots, \xi(f_n))=\sum_{i=1}^n \partial _i F(\xi(f_1), \dots, \xi(f_n)) f_i.
\end{equ}
The  Malliavin derivative operator $D: L^p(\Omega)\to L^p(\Omega; \cH)$, with initial domain $\cS$ is closable for any $p> 1$. 
We define $\bD^{1,p} $ as the completion of smooth random variables $\cS$ under the norm
\begin{equ}
	\| X \|^p_{\bD^{1,p}} \eqdef 
	\E \big( |X|^p \big)
	+ \E \big( \| DX\|_\cH^p  \big).  
\end{equ} 
In all the cases we consider, our random variables $X$ will be such that their Malliavin derivatives $DX$
have natural function-valued representatives, and we
denote their pointwise evaluations by $D_{t, x} X$. For example, for a deterministic 
 function $h$, $D_{r,z}  \xi(h) = h(r,z)$.
The $L^2$ adjoint of $D$ is denoted by $\delta$ and is known as the divergence operator or Skorokhod integral.
It is shown in \cite[Prop.~1.5.4]{nualart2006malliavin} that, for every $p \ge 2$, $\delta$ extends to a bounded operator
$\delta: \bD^{1,p}(\cH) \to L^p(\Omega)$. For a deterministic $f\in \cH$, we have $\delta(f) = \xi(f)$ and  $D (\xi(f))=f$. If  $w \in \dom \, \delta$ is $\F_t$-adapted, $\delta(w)$ coincides with the It\^o integral. We refer to e.g.\ \cite{nualart2006malliavin} for more details on Malliavin calculus. 


Before stating the lemma, we recall some covariance estimates -- used for the a priori bounds on the moments below and Theorems~\ref{thm:main-theorem-KPZ}-\ref{theo:homogenisation} in next sections.

\begin{lemma} \label{lem:gaussian_cov_est}
If $R(x) \lesssim (1 + |x|^\kappa)^{-1}$, where  $\kappa \in (2, d)$, then
uniformly in $t > 0$ and in $x \in \R^d$, 
	\begin{equs}
		\int_{\R^{d}} P_t(x - z)  R(z) \, dz &\lesssim (1 + |x| + \sqrt{t})^{- \kappa},\qquad \label{eq:time_decay_est} \\
		\int_0^\infty \int_{\R^{d}} P_{r}(z-x) R(z) \,dr \,dz
		& \lesssim 1 \wedge |x|^{2-\kappa}, \qquad
		\label{eq:control_space}\\
		\int_0^t \int_{\R^{2d}}  P_r(z_1-x) P_{t-r}(z_2) R(z_1 -z_2)\,dr \,dz_1 \, dz_2 &\lesssim  t(1+|x|+\sqrt t)^{-\kappa}. \qquad
			\label{eq:both_ends_cov}
	\end{equs} 
\end{lemma}

\begin{proof}
	The first two estimates follow from the proof of \cite[Lem.~2.3]{GHL23_arxiv_preprint}. By a change of variables, the left hand side of \eqref{eq:both_ends_cov}  equals
	$$\int_0^t \int_{\R^d} P_t(y-x)R(y) \, dr\, dy
	\lesssim t(1+|x|+\sqrt t)^{-\kappa}\;,$$
so that \eqref{eq:both_ends_cov} follows from \eqref{eq:time_decay_est}.
\end{proof}


	
%

In the following we gather moment estimates on $u$, $\vec Z$ and the Green's function $w_{s,y}$
appearing in Theorem~\ref{theo:homogenisation}. Crucially, for sufficiently small $\beta$, the negative moments of $u$ are also uniformly bounded.

\begin{lemma}\label{lem:moments} 
	Assume that $R$ satisfies Assumption~\ref{assump-cov-KPZ}. For any $p \geq 2$, there exist constants $\beta_0(p)>0$ and $C_p$,
	such that if $\beta \leq \beta_0$, the following holds for the solution of (\ref{eq:basic_linear_mSHE}).
	\begin{enumerate}
		\item\label{i} The solution of (\ref{eq:basic_linear_mSHE}) has positive and negative moments of all orders:
		\begin{equ}
			\sup_{s \geq 0,  x\in \R^d } \| u(s,x)\|_{p}
			\leq C_p\;,\quad
						\sup_{s\geq 0,  x\in \R^d }	\|  u(s,x)^{-1} \|_p \leq C_p\;.
			 \label{eq:moment_u}
		\end{equ}
		\item\label{ii} For all $x \in \R^d$ and $t > 0$, $u(t,x) \in \bD^{1,p}$. The Malliavin derivative of $u$ is given by
		\begin{equ}[e:MallDeru]
			D_{s,y}u(t,x) = \beta u(s,y)w_{s,y}(t,x) \1_{[0,t]}(s) \;,
		\end{equ}	
		and the process $w$ satisfies the pointwise estimate
		\begin{equ}[e:bound_w]
			\| w_{s,y}(t,x)\|_{p} \lesssim P_{t-s}(x-y)\;.
		\end{equ}
		\item\label{iii}The process $C_{s,y}(t)$ defined in \eqref{e:defn_Csy} satisfies $\sup_y\sup_{t>s\geq0}\| C_{s,y}(t) \|_p \lesssim 1$ and 
		\begin{equ}[e:estimate_Csy]
			\| C_{s,y}(t) - C_{s,y}(\tau)  \|_p \lesssim \beta (1 + (t\wedge \tau )- s)^{\frac {2 - \kappa } 4}\;.
		\end{equ}
	\end{enumerate}
	Assertion \ref{i} 
	also holds if we replace $u$ by $\vec Z$, and one has
\begin{equ}[e:MallDerZ] 
D_{r,z} \vec Z(t,x) = \beta \vec Z(r,z) w_{r,z} (t, x)\1_{(-\infty, t]}(r)\;.
	\end{equ}
\end{lemma}
\begin{proof}

The positive moment bounds in \eqref{eq:moment_u} follow from \cite[Lem.~2.8]{GHL23_arxiv_preprint}. Concerning the Malliavin derivative, we note that by the mild solution formulation we have
$u(t,x)=1+ \beta \delta( v_{t,x})$, where 
\begin{equ}
	v_{t,x}(r,z)  =  P_{{t} -r } ( x -z) u(r,z) \1_{[0,t]}(r)\;.
\end{equ}
Differentiating on both sides and using the commutation relation \cite[Eq.~1.46]{nualart2006malliavin}, 
it follows that for $s \in [0, t]$ one has
\begin{equ}
	D_{s,y} u(t,x) = \beta P_{{t} -s } ( x -y) u(s,y) + \beta \int_s^t \int_{\R^d}   P_{{t} -r } ( x -z)  D_{s,y} u(r,z) \, \xi(dr, dz)\;.
\end{equ}
This shows that $D_{s,y} u$ solves the same equation as $u$ itself, but with initial condition
$\beta u(s,\cdot)\delta_y$ at time $s$, so that \eqref{e:MallDeru} follows by linearity.

The moment bound on $w_{s,y}$ follows by the same steps as in the proof of \cite[Lem.~2.8]{GHL23_arxiv_preprint} (we omit tracking the same constants here). Consider $w^{(0)}_{s, y} (t,x) := P_{t-s}(x-y)$, and 
\begin{equ}
	w^{(n+1)}_{s,y}(t,x) = P_{t-s}(x-y) + \beta \int_s^t \int_{\R^d} P_{t-r}(x-z) w^{(n)}_{s,y}(r,z)\,\xi(dr,dz)\;.
\end{equ}
Letting $\tilde w^{(n)}_{s,y}(t,x) =  w^{(n+1)}_{s,y}(t,x) - w^{(n)}_{s,y}(t,x)$, by \eqref{eq:BDGineq+} one has 
\begin{equs}
	& \|\tilde w^{(n)}_{s,y}(t,x)\|_p\\
	&\leq \beta C_p \biggl( \int_{s}^{t}\int_{\R^{2d}}  \Big(\prod_{i = 1}^2 P_{t-r} (x-z_i) \|\tilde w^{(n-1)}_{s,y}(r,z_i)\|_p  \Big) \big| R(z_1 -z_2)\big| \, dz \, dr \biggr)^{\f 1 2}\\
	& \lesssim   \beta C_p \sup_{r > s }\sup_{z} \|w^{(n-1)}_{s,y}(r,z)\|_p\;.
\end{equs}
Thus, for sufficiently small $\beta \leq \beta_0(p)$, $ w^{(n)}_{s,y}$ converges to $ w_{s,y}$ as $n \to \infty$. Moreover, again by \eqref{eq:BDGineq+}, we have
\begin{equs}
	\|  w^{(n+1)}_{s,y} (t,x) \|^2_p &\lesssim   P_{t-s}(x-y) +  \\	
	& \quad \beta^2 \int_{s}^{t} \int_{ \R^{2d}} 
	R(z_1-z_2) \Big( \prod_{i=1}^{2} P_{t-r}(x-z_i) 
	\| w^{(n)}_{s,y} (s, z_i) \|_p  \Big)\,
	dz\, dr\;.
\end{equs}
Given these bounds and Lemma~\ref{lem:gaussian_cov_est}, 
\eqref{e:bound_w} follows by \cite[Lem.~2.7]{GHL23_arxiv_preprint}.  

	In view of Proposition~\ref{prop:C_Z_relation}, the bounds on the mass process
	$C$ are an immediate consequence of \eqref{eq:Lp_cvg_rate} in the uniform moment bounds on $\vec Z$.
	
	The uniform control on the negative moments of $u$ in \eqref{eq:moment_u} is a consequence of 
	the fact that $\log u(t,x)$ has sub-Gaussian left tails, namely for $\theta > 0$, 
	\begin{equ}\label{eq:left_tails}
		\mathbb{P}\bigl( \log u(t,x) \leq -\theta\bigr) \leq C e^{-\frac {\theta^2} 2}\;.
	\end{equ}
	The proof of \eqref{eq:left_tails} follows by \cite[Thm.~1.3]{cosco2020renormalizing},  given that $u$ is uniformly $L^2$ bounded for any $\beta \leq \beta_0(2)$. Their strategy to discretize and cut-off of the mollifier $\phi$ does not require $\phi$ to be compactly supported. Then, all steps carry through once noted that even with long range covariance $R$, uniformly in $t>0$,
	\begin{equs}
		\mathbf{E}_{B^{(1)}, B^{(2)}} \Big[ &\int_0^t R(B^{(1)}_s -  B^{(2)})\, ds \; e^{ \beta^2 \int_0^t R(B^{(1)}_s -  B^{(2)})\, ds} \Big] 
		\lesssim 
		\sup_{\substack{ \beta \leq \beta_0(2), \\ t>0}}
		\E [u(t, 0)^2] \lesssim 1, 
	\end{equs}
	where $\mathbf{E}_{B^{(1)}, B^{(2)}}$ denotes the expectation with respect to Brownian motions $B^{(i)}$ independent of $\xi$. The first inequality above follows by noting that, via Feynman-Kac formula, one has an explicit expression for the second moment of $u$ (cf. \cite[Lem.~3.1]{MSZ16_weak_strong_mshe}): 
	\begin{equ}
		\E [u(t, 0)^2] = \mathbf{E}_{B^{(1)}, B^{(2)}} \Big[ e^{ \beta^2 \int_0^t R(B^{(1)}_s -  B^{(2)})\, ds} \Big]. 
	\end{equ}
	
	Regarding $\vec Z$, we note that it solves the fixed point problem
	\begin{equ}[e:Z_mild_eqn]
		\vec Z(t,x)=1 + \beta \int_{-\infty}^t\int_{\R^d} P_{t-s}(x-y) \vec Z(s,y)\,\xi(ds,dy)\;,
	\end{equ}
	whence we conclude that
	\begin{equ}
	D_{r,z} \vec Z(t,x) =\beta P_{t-r}(x-z) \vec Z(r,z)+ \beta \int_{-\infty}^t\int_{\R^d} P_{t-s}(x-y) D_{r,z} \vec Z(s,y)\,\xi(ds,dy).
	\end{equ}
Since $D_{r,z} \vec Z(t,x)=0$ for $r>t$,   assertion (\ref{e:MallDerZ}) holds.
	Since the bounds in \ref{i} are uniform in time, their extension to $\vec Z$ is immediate. It remains to show (\ref{e:estimate_Csy}), which follows from
	\Cref{prop:C_Z_relation}, the fact that $\cev Z(s,y)(\xi)=\lim_{r\to -\infty}u^{(r)}(-s,y)(\CR \xi)$, and \Cref{thm:stat_field}.
\end{proof}

With positive and negative moment bounds in place, we can Malliavin differentiate transformations of $u$ and then approximate fluctuations through the homogenisation result. 

\begin{lemma}\label{lem:transform_moments+decorrelation}
	Assume $\gfr \in \C^1(\R_+)$ and that, for some $M, q \geq 1$, satisfies for all $z \in \R_+$,
	\begin{equ} \label{eq:control_growth_Psi}
		|\gfr(z)| + | \gfr'(z) | \leq M(z^{-q} + z^q). 
	\end{equ}
	Then, for any $p >1$ there exists $\beta_0^{M,q}(p) = \beta_0(p, d, R, q, M)$ such that for any $ \beta \leq \beta_0^{M,q}(p)$, 
	we have $\gfr (u(t,x)) \in \bD^{1,p}$ with
	\begin{equ}\label{eq:malliavin_der_g}
		D_{s,y} [\gfr (u(t,x)) ]= \beta  \gfr' (u(t,x)) u(s,y)w_{s,y}(t,x) \1_{[0,t]}(s), 
	\end{equ}
	and the following bounds are satisfied:
	\begin{equs}
		\| \gfr(u(t,x))\|_p  &\lesssim M \;, &\quad 	
		\| \gfr'(u(t,x)) \|_p  &\lesssim M\;, \qquad \label{eq:moment_gu} \\
		\| D_{s,y}\gfr(u(t,x))\|_{p} &\lesssim  M P_{t-s}(x-y)\;, &\quad  
		\| \gfr (u(t,x)) \|_{\bD^{1,p}} &\lesssim M\;, \qquad \label{eq:moment_Dgu}
	\end{equs}
	uniformly over $s$, $t$, $x$, $y$.
	Moreover, uniformly over  $x \in \R^d$ we have  
	\begin{equ}\label{eq:g_cov_space_estimate}
		\sup_{t>0} \big| \cov\left( \gfr(  u(t, x)) , \gfr( u(t, 0) )  \right) \big| \lesssim  M \left( 1 \wedge |x|^{2 -\kappa}\right)\;.
	\end{equ}	
	All the estimates above also hold with $\vec Z(t,x)$ in place of $u(t,x)$. Finally, we also have the following convergence rate 
	\begin{equ} \label{eq:Lp_rate_transform}
		\|	\gfr( u^{(s)}(t,x)) - \gfr(\vec Z (t,x) )\|_p \lesssim  
		1 \wedge (t-s)^{\frac {2-\kappa} 4}\;.
	\end{equ}
\end{lemma}

\begin{proof} For $\beta\leq\beta_0(pq)$, bounds \eqref{eq:moment_gu} follow immediately from  \eqref{eq:moment_u} 
and the assumptions on $\gfr$.
	Observe that $\beta_0(p)$ is non-increasing in $p$.
	
	Regarding the Malliavin derivative, let us introduce an approximation by bounded functions $\gfr_n \in \C_b^1(\R_+)$ such that $\gfr_n = \gfr$ on $[\frac 1 n, n]$. 
	Consider $\varrho_n : \R_+ \to \R_+$ defined by
	\begin{equ} 	
		\varrho_n(x) := \begin{cases}
			\frac 1 {2n} + \frac n 2 x^2 \qquad  &\text{for } x \in [0, \frac 1 n], \\
			x\qquad  \qquad  &\text{for } x \in [\frac 1 n, n], \\
			\frac 3 2 n - \frac 1 {2n} (x - 2n)^2  \qquad   &\text{for } x \in [n, 2n], \\
			\frac 3 2 n \qquad  &\text{for } x \in [2n, +\infty). 
		\end{cases}
	\end{equ}
	Then, we let $\gfr_n := \gfr \circ \varrho_n$. Note that $\varrho_n \in  \C_b^1(\R_+)$ with $\sup_{x}|\varrho_n'(x)| \leq 1$ and range $\varrho_n (x) \in [\frac 1 {2n}, \frac 3 {2} n]$. Thus, $\gfr_n \in \C_b^1(\R_+)$ for all $n$. Moreover, we note that uniformly in $n$:
	\begin{equ}
		|\varrho_n (x)| \leq |x| +1, \qquad |\varrho_n (x)|^{-1} \leq \tfrac 1 {|x|} + 1. 
	\end{equ}
	Then, given growth control \eqref{eq:control_growth_Psi} of $\gfr$, for $\beta\leq\beta_0(pq)$ we can also infer 
	\begin{equ}
		\sup_{n >0, s \geq 0,  x\in \R^d } \| \gfr_n(u(s,x))\|_p  \lesssim 1 , \qquad 	
		\sup_{n >0, s \geq 0,  x\in \R^d } \| \gfr_n'(u(s,x))\|_p  \lesssim 1.
	\end{equ}


	Since $u(t,x) \in \bD^{1,p}$ by Lemma~\ref{lem:moments}, we have  
	$\gfr_n (u(t,x)) \in \bD^{1,p}$ with $D  \gfr_n (u(t,x))  = \gfr_n' (u(t,x)) D u(t,x)$  by \cite[Prop. 1.2.3]{nualart2006malliavin}. We let
	\begin{equs}
		F_{r,z}^{(n)}(t,x) &=  D_{r,z}  \big(\gfr_n (u(t,x)) \big) - \gfr' (u(t,x)) D_{r,z} u(t,z) \\
		& = \1_{ \{  \frac 1 n \leq u(t,x) \leq n \}^c} ( \gfr_n' (u(t,x)) - \gfr' (u(t,x)) )   D_{r,z} u(t,x).
	\end{equs}
	Then, for sufficiently small $\beta$ we use the moments estimates on $u,  \gfr', \gfr_n'$ and $Du$, 
	\begin{equs}
		\| F_{r,z}^{(n)} (t,x) \|_p &\leq \mathbb{P}( \{ u< \tfrac 1 n\} \cup \{ u>n\})^{\frac 1 {3p}}  \|  \gfr_n' (u(t,x)) - \gfr' (u(t,x)) \|_{3p} \|  D_{r,z} u(t,x) \|_{3p}\\
		&\lesssim  n^{-\frac 1 p} M  P_{t-r}(x -z).  \label{eq:decay_diff_p}
	\end{equs}
	By Minkowski, H\"{o}lder's inequalities and \eqref{eq:control_space},
	\begin{equs}
		\E \big[ \|F ^{(n)}(t,x) \|_\cH^{p}\big] &= \E \Big[ \Big(  \int_{0}^{t}\int_{ \R^{2d}} \prod_{i=1}^{2} F_{r,z_i}^{(n)}(t,x) R(z_1 -z_2) \,dr \,dz\Big)^{\frac p 2} \Big]	\\
		& \leq  \Big(  \int_{0}^{t}\int_{ \R^{2d}} \| \prod_{i=1}^{2} F_{r,z_i}^{(n)}(t,x) \|_{\frac p 2} R(z_1 -z_2) \,dr \,dz\Big)^{\frac p 2}\\
		& \leq n^{-1}  M^p \Big(  \int_{0}^{t}\int_{ \R^{2d}}  \prod_{i=1}^{2} P_{t-r}(x-z_i)  R(z_1 -z_2) \,dr \,dz\Big)^{\frac p 2}\\ &\leq C M^p n^{-1}.
	\end{equs}
	Thus, $F^{(n)}(t,x) \to 0$ in $L^p(\Omega, \cH)$ as $n \to \infty$. On the other hand, $\gfr_n(u(t,x))  \to \gfr(u(t,x))$ in $L^p(\Omega)$. Since $D$ is closed, we conclude $\gfr(u(t,x)) \in \dom(D)$ and  $D_{s,y} \gfr(u(t,x)) = \gfr'(u(t,x)) D_{s,y} u(t,x)$. Then, given \eqref{e:MallDeru}, representation \eqref{eq:malliavin_der_g} holds. Moreover, for sufficiently small $\beta$, given \eqref{e:bound_w} and the uniform moments bounds on $\gfr'(u)$ and $u$, we have $\| D_{r,z}\gfr(u(t,x))\|_{p} \lesssim  M P_{t-r}(x-z)$. Since the inequality constant is independent of $(t,x)$, we estimate similarly to above 
	\begin{equs}
		\E \big[ \| D \gfr(u(t,x))(t,x) \|_\cH^{p}\big] &\leq 
		\Big(  \int_{0}^{t}\int_{ \R^{2d}}  \prod_{i=1}^{2} \| D_{r,z_i}\gfr(u(t,x)) \|_{p} \, R(z_1 -z_2) \,dr \,dz\Big)^{\frac p 2}\\
		& \leq \Big( \int_{0}^{\infty}\int_{ \R^{2d}}  \prod_{i=1}^{2} P_{t-r}(x-z_i) R(z_1 -z_2) \,dr \,dz \Big)^{\frac p 2} \lesssim 1,  
	\end{equs}
	where we used \eqref{eq:control_space}. This concludes the proof of \eqref{eq:moment_Dgu}.
	
	%
	
	Using the Clark--Ocone formula in the same way as in  \cite[Lem.~2.12]{GHL23_arxiv_preprint}, we obtain the bound
	\begin{equs}
		\bigl|\cov&\left( \gfr(  u(s, x)) , \gfr( u(s, 0))  \right)\bigr| 
		 \lesssim M \int_0^t \int_{\R^{2d}} 
		P_{t-r}(x-z_1) P_{t-r}(z_2) R(z_1-z_2) \,dz \,dr,
	\end{equs}
	so that \eqref{eq:g_cov_space_estimate} follows by \eqref{eq:control_space}.  
	
	It remains to show \eqref{eq:Lp_rate_transform}, for which we make use of the identity
	\begin{equ}
		\gfr( u^{(s)}(t,x)) - \gfr(\vec Z (t,x) ) = \big( u^{(s)}(t, x) - \vec Z(t,x)\big) \int_{0}^1 \gfr'\big( \theta u^{(s)}(t, x) + (1-\theta) \vec Z(t,x)\big) \, d\theta\;.
	\end{equ}
Using the fact that $u$ is non-negative and the uniform a priori bounds on positive and negative moments of $u$ and $\vec Z$, one obtains a uniform $L^{2p}$ bound on $ \int_{0}^1 \gfr'\big( \theta u(t,x) + (1-\theta) Z(t,x)\big) d\theta$ for sufficiently small $\beta$. Thus, \eqref{eq:Lp_cvg_rate} implies
	\begin{equ} 
		\|	\gfr( u^{(s)}(t,x)) - \gfr(\vec Z (t,x) )\|_p \lesssim M \big\|u^{(s)}(t, x) - \vec Z(t,x)\big\|_{2p} \lesssim 1 \wedge (t-s)^{\frac {2-\kappa} 4}\;,
	\end{equ}
which is precisely \eqref{eq:Lp_rate_transform}, concluding the proof. 
\end{proof}

\section{The main homogenisation theorem}
\label{sec:Homogenisation}

In view of Theorem~\ref{theo:homogenisation}, recalling that $\vec Z(t,x) = \lim_{s\to -\infty} u^{(s)}(t,x)$, we define
\begin{equ}[e:defn_rho]
	\rhosy(t,x) =  \f{w_{s,y}(t,x)} {P_{t-s}(x-y)} -  C_{s,y}(t)\vec Z(t,x)\;.
\end{equ}
By Proposition~\ref{prop:C_Z_relation} and \eqref{e:estimate_Csy}, we know that
\begin{equ}
	\|  C_{s,y}(t) - \cev Z (s, y) \|_p \lesssim \beta (1 + t - s)^{\f {2- \kappa} 4}\;.
\end{equ}
Since $\kappa>2$,  our main homogenisation result then follows from the following proposition. 
The proof relies on a kernel estimate and moment bounds on Skorokhod correction terms proved below. 

\begin{proposition}\label{prop:rho_homogenisation}
Let  $\mu= \f 12 -\f 1\kappa$.
Then, given $p\geq 1$, there exists $\tilde \beta(p) > 0$, such that for $\beta \leq \tilde \beta$, 
 the following holds
\begin{equ}[e:homo_estimate]
	\sup_{x,y} \sup_{t > s \geq 0} \frac{ \norm{\rhosy(t,x) }_p }{ (1 + t-s)^{-\mu} (1 + (1 + t-s)^{-\f 1 2}|x-y|) } \lesssim \beta\;.
\end{equ}
\end{proposition}

\begin{proof}
Let us consider $\tilde w_{s,y}(t,x) = w_{s,y}(t,x)/P_{t-s}(x-y)$. 
Since $w_{s,y}(t,x)$ solves 
\begin{equ}[e:defn_w]
	w_{s,y}(t,x) = P_{t-s}(x-y) + \beta \int_s^t \int_{\R^d} P_{t-r}(x-z)w_{s,y}(r,z)\,\xi(dr,dz)\;,
\end{equ}
one has the identity
\begin{equ}[e:exprTildew]
	\tilde w_{s,y}(t,x) = 1 + \beta\int_s^t \int_{\R^d} 	q_{s, t}^{x,y} (r, z)   \tilde w_{s,y}(r,z)\,\xi(dr,dz)\;, 
\end{equ}	
where $q_{s, t}^{x,y}$ is kernel of the Brownian bridge starting from $s$ at $x$ and ending at $y$ at time $t$,
\begin{equ} [e:brownian_bridge]
	q_{s, t}^{x,y} (r, z) =\f {	P_{r-s}(z-y) P_{t-r}(x-z)} {P_{t-s}(x-y)}  = P_{(r-s)(t-r)/(t-s)} (z-y - \tfrac{r-s}{t-s}(x-y))\;. 
\end{equ}
On the other hand, integrating \eqref{e:defn_w} over $x$,  we obtain
\begin{equ}[e:Csy_eqn]
	C_{s,y}(t)	=  1 + \beta \int_s^t \int_{\R^d} P_{r-s} (z-y) \tilde w_{s,y}(r,z) \xi(dr,dz)\;.
\end{equ}
As a consequence, we can wrte $\rho_{s,y}$ as 
 \begin{equs}[e:rho_id1]
	\rho_{s,y}(t,x)  &= \tilde w_{s,y}(t,x) - C_{s,y}(t) \vec Z(t,x)\\
	&= 1 + \beta \int_{s}^t \int_{\R^d} 	q_{s, t}^{x,y} (r, z) \tilde w_{s,y}(r,z)\,\xi(dr,dz) \\
	&\quad - \Bigl(1 + \beta \int_{-\infty}^t \int_{\R^d} P_{t-r}(x-z) \vec Z(r,z)\,\xi(dr,dz)\Bigr) C_{s,y}(t)\\
	&= \beta\int_s^t \int_{\R^d} \big( 	q_{s, t}^{x,y} (r, z) - P_{r-s}(y-z)\big) \tilde w_{s,y}(r,z)\,\xi(dr,dz) \\
	&\quad - \beta \int_{-\infty}^t \int_{\R^d}P_{t-r}(x-z) \vec Z(r,z)\,\xi(dr,dz) \, C_{s,y}(t)\;.
\end{equs}
To treat the term $ C_{s,y}(t)$, we split the last integral into two parts, integrating over $r \in (-\infty, s)$ leads to a contribution proportional to
\begin{equ}[eq:defn_E1]
	E_1 := \int_{-\infty}^s \int_{\R^d}P_{t-r}(x-z)\vec Z(r,z)\,\xi(dr,dz)\, C_{s,y}(t)\;.
\end{equ}
For the other region $r \in [s, t]$ , we rewrite it in terms of a Skorokhod integral using the fact  (cf. \cite[Prop.~1.3.3]{nualart2006malliavin})
that 
$F \int G\,d\xi = \delta (FG) + \langle  G , DF \rangle_\cH$ 
with $F = C_{s,y}(t)$ and $G(r,z) = P_{t-r}(x-z)\vec Z(r,z)\1_{[s, t]}(r)$. Moreover, $C_{s,y}(r)$ is $\F_r$-adapted and we obtain, writing $\delta \xi$ for the Skorokhod integration:
\begin{equs}
&\int_{s}^t \int_{\R^d} P_{t-r}(x-z) \vec Z(r,z)\, \xi(dr,dz) \, C_{s,y}(t)  \\
& \quad = \int_{s}^t \int_{\R^d}  P_{t-r}(x-z) \vec Z(r,z)  C_{s,y}(t) \, \delta \xi (dr, dz) \\
&\quad \quad +\int_{s}^t \int_{\R^{2d}}  P_{t-r}(x-z_1) \vec Z(r,z_1)\, D_{r,z_2} C_{s,y}(t)R(z_1-z_2 ) \,dr\,dz \\
& \quad =  \int_{s}^t \int_{\R^d}  P_{t-r}(x-z) \vec Z(r,z)  C_{s,y}(r) \, \xi (dr, dz) \\
& \quad \quad + \int_{s}^t \int_{\R^d}  P_{t-r}(x-z) \vec Z(r,z)  \big( C_{s,y}(t) - C_{s,y}(r)\big)  \, \delta \xi (dr, dz)\\
&\quad \quad +\int_{s}^t \int_{\R^{2d}}  P_{t-r}(x-z_1) \vec Z(r,z_1)\, D_{r,z_2} C_{s,y}(t)R(z_1-z_2 ) \,dr\,dz\;.  
\end{equs}
Bringing all these back to \eqref{e:rho_id1} and given the definition \eqref{e:defn_rho} of $\rho_{s,y}$, we obtain 
\begin{equs}[e:rho_decomp]
	\rho_{s,y}(t,x)  =  \beta \int_{s}^t \int_{\R^d}  P_{t-r}(x-z) \rho_{s,y}(r,z)  \, \xi (dr, dz) + \beta E_0 - \beta \sum_{i=1}^3 E_i\;,\quad 
\end{equs}
where $E_1$ was defined in \eqref{eq:defn_E1} and the other terms are given by 
\begin{equs}
	E_0 &:= \int_s^t \int_{\R^d} \big( 	q_{s, t}^{x,y} (r, z) - P_{r-s}(y-z) - P_{t-r}(x-z) \big) \tilde w_{s,y}(r,z)\,\xi(dr,dz)\;,\\
	E_2 &:=  \int_{s}^t \int_{\R^d}  P_{t-r}(x-z) \vec Z(r,z)  \big( C_{s,y}(t) - C_{s,y}(r)\big)  \, \delta \xi (dr, dz)\;, \\
	E_3 &:=  \int_{s}^t \int_{\R^{2d}}  P_{t-r}(x-z_1) \vec Z(r,z_1)\, D_{r,z_2} C_{s,y}(t)R(z_1-z_2 ) \,dr\,dz\;.
\end{equs}
We can now estimate their $L^p$ norms. Given a time interval $I = [a,b]$ and an adapted random process $f(s,y)$, by Burkholder--Davis--Gundy (BDG) inequality, followed by Minkowski and H\"{o}lder inequalities, we have the following estimate 
\begin{equs}
{}&\Big\|\int_{I\times \R^d} f(s,y) \xi(ds, dy)\Big\|_{p} 
\leq C_p \E \left[ \Big(\int_I\int_{\R^{2d}} f(s,y_1)f(s,y_2) R(y_1 - y_2)  \,dy \,ds\Big)^{\f p 2 } \right]^{\f 1 p}\\
		& \qquad \qquad \leq C_p \Big( \int_I\int_{\R^{2d}} \|f(s,y_1) \|_p \|f(s,y_2)\|_p R(y_1 - y_2)  \,dy \,ds \Big)^{\f 1 2}\;.
		\label{eq:BDGineq+}
\end{equs}
Given the uniform moments bounds, on $C_{s,y}$ and $\vec Z$, from Lemma~\ref{lem:moments} and the kernel estimate \eqref{eq:time_decay_est}, for $\beta \leq  \beta_0(2p)$ we have 
\begin{equs}
	\| E_1 \|_p &\lesssim \Big\| \int_{-\infty}^s P_{t-r} (x - z)\vec Z(r,z)\, \xi(dr, dz)   \Big \|_{2p} \| C_{s,y} (t) \|_{2p} \\ 
	&\quad \lesssim  \Big( \int_{-\infty}^s \int_{\R^{2d}}  \prod_{j=1}^2 \Big( P_{t-r} (x-z_j)  \| \vec Z(r,z_j) \|_{2p} \Big) R(z_1 -z_2) \,dr \,dz \Big)^{\f 1 2} \\
	&  \quad \lesssim \Big( \int_{-\infty}^s (1 + t -r)^{-\f \kappa 2 } dr \Big)^{\f 1 2} \lesssim (1 + t-s)^{\frac {2-\kappa} 4}\;.
\end{equs}
For $E_0$, we note that by \eqref{e:bound_w},
\begin{equ}
	\sup_{x, y} \sup_{t > s \geq 0} \| \tilde w_{s,y} (t,x) \|_p \lesssim 1\;. 
\end{equ}
Then, estimating moments like above, followed by change of variables $z_i \mapsto z_i - y$, $r \mapsto r -s$, and Lemma~\ref{lem:Kernel_P_bound} below, we can bound $E_0$ as follows for $\beta \leq \beta_0(p)$:
\begin{equs}
	\|  &E_0 \|_p \leq  C_p\Big( \int_s^t \int_{\R^{2d}}\prod_{i = 1}^2  \Big(  	q_{s, t}^{x,y} (r, z_i) - P_{r-s}(y-z_i) - P_{t-r}(x-z_i) \Big) R(z_1 - z_2) \,dr \,dz \Big)^{\f 1 2} \\
	& \; =C_p \Big( \int_0^{t-s} \int_{\R^{2d}} \prod_{i = 1}^2 \Big(  q_{0, t-s}^{x-y,0} (r, z_i) - P_{r-s}(z_i) - P_{t-r}(x-y-z_i) \Big) R(z_1 - z_2) \,dr \,dz \Big)^{\f 1 2}\\
	& \;\lesssim (1 + t-s)^{-\mu}
	(1 + (1 + t-s)^{-\f 1 2}|x-y|)\;. 
\end{equs}
By Lemma~\ref{lem:Skorokhod_terms} below, for $\beta \leq  \beta_1(3p)\wedge \beta_0(3p)$ also the terms $E_2$ and $E_3$  satisfy 
\begin{equ}
\| E_i \|_p \lesssim (1 + t-s)^{-\mu} \;, \qquad i = 2,3 \;. 
\end{equ}  
Among the estimates, the worst decay rate is the bound on $E_0$, and we obtain
\begin{equs}
	\sum_{i= 0}^3 \| E_i \|_p  \lesssim m(t-s, x-y)\;,
\end{equs}
where we denote the space-time weight 
\begin{equ}
	m(r, z) = (1 + r)^{-\mu} (1 + (1 + r)^{-\f 1 2}|z|)\;.
\end{equ}

Finally, let us consider the weighted norm
\begin{equ}
	M := \sup_{x,y} \sup_{t > s \geq 0} \frac{ \norm{\rhosy(t,x) }_p }{ m(t-s, x-y) }\;. 
\end{equ} 
We estimate the first term in \eqref{e:rho_decomp} as follows:
\begin{equs}
&\Big\| \int_{s}^t \int_{\R^d}  P_{t-r}(x-z) \rho_{s,y}(r,z)  \, \xi (dr, dz) \Big\|_p \\
& \quad \leq C_p \Big(   \int_{s}^t  \int_{\R^{2d}} \prod_{i = 1}^2  \big(  P_{t-r}(x-z_i) \| \rho_{s,y}(r,z_i) \|_p   \big)  R(z_1 - z_2 ) \,dr \,dz \Big) ^{\frac 1 2 } \\ 
& \quad  \lesssim M \Big(   \int_{s}^t (1+r-s)^{-2\mu} \int_{\R^{2d}}
\Bigl(\prod_{i = 1}^2  P_{t-r}(x-z_i)  \Bigr)  R(z_1 - z_2 )  \,dr \,dz \Big) ^{\frac 1 2 } \\ 
& \quad \quad + M \Big( \sum_{j =1}^2  \int_{s}^t  (1+r-s)^{-1-2\mu}  \int_{\R^{2d}} \Bigl(\prod_{i = 1}^2 |z_j- y |^2 P_{t-r}(x-z_i) \Bigr)   R(z_1 - z_2 )  \,dr \,dz \Big) ^{\frac 1 2 } \\
& \quad \lesssim M \Big(   \int_{s}^t (1+r-s)^{-2\mu} (1 + t-r)^{-\f \kappa 2}\, dr \Big) ^{\frac 1 2 } \\
& \quad \quad + M \Big( \int_{s}^t (1+r-s)^{-1-2\mu} (1 + t-r)^{-\f \kappa 2} \big( | x - y|^2 + (t-r) \big)  \, dr dr \Big) ^{\frac 1 2 }\;, \label{e:rho_bound}
\end{equs}
where we used \eqref{eq:time_decay_est} and the fact that $\int_{ \R^d} |z|^2 P_r(z-x) dz  \lesssim r + |x|^2$ for the second integral. 
For $\alpha, \beta > 0$, we note
\begin{equ}[e:alpha_beta_decay]
	\int_0^T (1+\tau)^{-\alpha} (1 + T -\tau)^{-1-\beta} \, d\tau  \lesssim (1+T)^{-\min(\alpha, \beta)}\;.
\end{equ}
This follows by splitting the integration region. On the interval $[0, T/2]$, we have $T -\tau \geq T/2$ and 
\begin{equ}
	\int_0^{\f T 2} (1+\tau)^{-\alpha} (1 + T -\tau)^{-1-\beta} \, d\tau \lesssim (1 + T )^{-1-\beta} T \lesssim (1 + T )^{-\beta}\;, 
\end{equ}
while on  $[T/2, T]$, we have 
\begin{equ}
	\int_{\f T 2} ^T (1+\tau)^{-\alpha} (1 + T -\tau)^{-1-\beta} \, d\tau \lesssim (1+T)^{-\alpha} \int_0^{\infty}(1 + r)^{-1-\beta} \, dr \lesssim  (1+T)^{-\alpha}\;.
\end{equ}
Then, going back to \eqref{e:rho_bound}, we apply \eqref{e:alpha_beta_decay} to both terms and obtain
\begin{equ}
	\Big\| \int_{s}^t \int_{\R^d}  P_{t-r}(x-z) \rho_{s,y}(r,z)  \, \xi (dr, dz) \Big\|_p \lesssim M \; m(t-s, x-y)\;.
\end{equ}
Hence, by the estimates above, for some $C > 0$ we have 
\begin{equs}
\| \rhosy(t, x) \|_p \leq C \beta M \, m(t-s, x-y) +  C \beta \, m(t-s, x-y)\;.
\end{equs} 
Taking  $\tilde \beta>0$ such that $C \tilde \beta < 1 $  and $\tilde \beta \leq \beta_1(3p)\wedge \beta_0(3p)$, we conclude \eqref{e:homo_estimate}. \end{proof}
\begin{lemma}\label{lem:Skorokhod_terms}
Let $p\geq1$. With notation as in the proof of Proposition~\ref{prop:rho_homogenisation}, for any $\beta \leq \beta_1(3p)\wedge \beta_0(3p)$, for $i = 2, 3$, we have    
the following uniformly in $x, y \in \R^d$ and $t>s\geq 0$:
	\begin{equs}
		\| E_2(s, t, y, x) \|_p &\lesssim  \beta \, (1 + t-s )^{\frac {2-\kappa} 4}\;, \\
		\| E_3(s, t, y, x) \|_p &\lesssim \beta \, (1 + t-s )^{1 - \frac {\kappa} 2}\;. 
	\end{equs}
\end{lemma}
\begin{proof}
Let us start with the Skorokhod correction term $E_3$. 
Similarly to Lemma~\ref{lem:moments}, where we obtained $D_{s,y}u(t,x) = \beta u(s,y)w_{s,y}(t,x)\1_{[0,t]}(s)$, 
one has
	\begin{equ}
		D_{r,z} w_{s,y}(t,x) = \beta w_{s,y}(r,z) w_{r,z}(t,x) \1_{[s, t]}(r)\;,
	\end{equ}
 so that, integrating over $x$,
	\begin{equ}[e:MallDerC]
		D_{r,z} C_{s,y}(t) = \beta w_{s,y}(r,z) C_{r,z}(t)\1_{[s, t]}(r)\;.
	\end{equ}
	Therefore, the term $E_3$ is given by
	\begin{equs}
	E_3=	\int_{s}^t \int_{\R^{2d}}  &P_{t-r}(x-z_1) \vec Z(r,z_1)\, D_{r,z_2} C_{s,y}(t)R(z_1-z_2 ) \,dr \,dz\\
		&=
		\beta \int_{s}^t \int_{\R^{2d}} P_{t-r}(x-z_1)\vec Z(r,z_1)\,w_{s,y}(r,z_2) C_{r,z_2}(t)R(z_1-z_2) \,dr \,dz\;.
	\end{equs}
	Using the uniform moments estimates on $\vec Z$, 
	$w$
	and $C$ from Lemma~\ref{lem:moments}, 
	by Minkowski and H\"{o}lder's inequalities, followed by \eqref{eq:both_ends_cov} we obtain for $\beta \leq \beta_0(3p)$: 
	\begin{equs}
		\| E_3 \|_p  &\leq \beta \int_{s}^t \int_{\R^{2d}} P_{t-r}(x-z_1)   \| \vec Z(r,z_1) w_{s,y}(r,z_2) C_{r,z_2}(t) \|_p  R(z_1-z_2) \,dr \,dz\\
		&\lesssim \beta \int_{s}^t  \int_{\R^{2d}} P_{t-r}(x-z_1) P_{r-s}(z_2-y)   R(z_1-z_2) \,dr \,dz \\ & \lesssim \beta (1 + t - s)^{1 - \f { \kappa } 2}\;.
	\end{equs}
	
	To estimate the $L^p$ norm of $E_2$, we use the continuity of the Skorokhod integral $\delta : \bD^{1,p}(\cH) \to L^p$ \cite[Prop.~1.5.4]{nualart2006malliavin}. Letting 
	\begin{equ}
	F_{s,t, x,y}(r,z)  =  P_{t-r}(x-z) \vec Z(r,z)  \big( C_{s,y}(t) - C_{s,y}(r)\big) \1_{[s, t]}(r) \;, 
	\end{equ}
	we have $E_2 = \delta (F_{s,t, x,y})$. Therefore, dropping subscripts of $F$ going forward, we have
	\begin{equ} [e:E2_norm_bound]
		\| E_2\|_p \lesssim \| F \|_{L^p(\Omega, \cH)} + \| D F \|_{L^p(\Omega, \cH \otimes \cH)}\;.
	\end{equ}
	We see that $F$ is a pointwise product of a Malliavin differentiable random variable $M$ and a deterministic element $h \in \cH$. Namely, $F(r,z) = M(r,z) h(r,z)$, where
	\begin{equ}
		M(r,z) = \vec Z(r,z)  \big( C_{s,y}(t) - C_{s,y}(r)\big)\;, \qquad h(r,z) = P_{t-r}(x-z) \1_{[s, t]}(r)\;.
	\end{equ}

	Given \eqref{e:estimate_Csy}, we note that 
	\begin{equs}
		\| M(r,z)  \|_p  &\lesssim   \| \vec Z(r,z) \|_{2p} \| C_{s,y}(t) - C_{s,y}(r) \|_{2p}  
		\lesssim \beta   (1 + r - s)^ {\f {2 - \kappa} 4} \;. 
	\end{equs}
	With this, using Minkowski inequality, \eqref{eq:time_decay_est} and \eqref{e:alpha_beta_decay}, we bound the first term: 
	\begin{equs}
		\| F &\|_{L^p(\Omega, \cH)} = \E \Big[  \big(\int_{\R}  \int_{\R^{2d}} F(r,z_1)F(r,z_2) R(z_1-z_2)\,dr\,dz \big)^{\f p 2}  \Big] ^{\f 1 p} \\
		& \;\leq \Big(  \int_{s}^t  \int_{\R^{2d}}\| M(r,z_1) M(r,z_2)  \|_{\f p 2} \, h(r,z_1) h(r,z_2)  R(z_1-z_2)\,dr\,dz  \Big)^{\f 1 2} \\
		& \;\lesssim \Big( \beta^2 \int_{s}^t (1 + r - s)^ {1 - \f  \kappa 2} \int_{\R^{2d}} P_{t-r}(x-z_1 ) P_{t-r}(x-z_2) R(z_1-z_2)\,dr\,dz \Big)^{\f 1 2} \\
		& \;\lesssim \beta   \Big( \int_{s}^t  (1 + r - s)^{1 - \f \kappa 2}   (1 + t - r)^ {- \f \kappa 2}  \,dr \Big)^{\f 1 2}  \lesssim  \beta (1 + t - s)^ { \f {2 - \kappa} 4}\;.
	\end{equs}
	For the derivative term, by Minkowski and H\"{o}lder's inequalities we have
	\begin{equs}
		\| &DF \|_{L^p(\Omega, \cH \otimes \cH)}\\
		 &= \E \Big[  \big( \int_{\R} \int_{\R^{2d}} \langle  D M(r, z_1) , D M(r, z_2) \rangle_{\cH}  \, \, h(r,z_1) h(r,z_2) R(z_1 - z_2) \, dz\, dr  \big)^{\f p 2}\Big] ^{\f 1 p} \\
		& \leq \Big( \int_{s}^t \int_{\R^{2d}}   \| \langle  D M(r, z_1) , D M(r, z_2) \rangle_{\cH} \|_{\f p 2} \, h(r,z_1) h(r,z_2) R(z_1 - z_2)  \, dz\, dr \Big)^{\f 1 2} \\
		& \leq \Big(\int_{s}^t \int_{\R^{2d}} \Bigl( \prod_{i = 1}^2 \|   D M(r, z_i) \|_{L^{p}(\Omega, \cH)} \, h(r,z_i)\Bigr) R(z_1 - z_2)  \, dz\, dr \Big)^{\f 1 2}\;. \\ 
	\end{equs}
	We note that $\vec Z$ and $C$'s are adapted, satisfying identities \eqref{e:MallDerZ} and \eqref{e:MallDerC} respectively. Then,  
	\begin{equs}
		D_{u, \bar z} M(r, z) &= D_{u, \bar z}\vec Z(r,z) \,  \big( C_{s,y}(t) - C_{s,y}(r)\big) +  \vec Z(r,z) \,  D_{u, \bar z} \big( C_{s,y}(t) - C_{s,y}(r)\big) \\
		& = \beta w_{u, \bar z} (r,z)  \vec Z(u, \bar z) \big( C_{s,y}(t) - C_{s,y}(r)\big) \1_{[s, r]}(u)  \\[0.5em]
		& \quad +  \beta \vec Z(r,z) w_{s,y} (u, \bar z)  \big( C_{u, \bar z}(t) - C_{u, \bar z}(r)\big) \1_{[s, r]}(u) \\
		& \quad + \beta  \vec Z(r,z) w_{s,y} (u, \bar z) C_{u, \bar z}(t)  \1_{[r, t]}(u) =: \beta \sum_{j = 1}^3 m^{(j)}_{r,z}(u,\bar z )\;.
	\end{equs}
	By Minkowski and H\"{o}lder inequalities 
	\begin{equ}[e:DM_bound]
		\|   D M(r, z) \|_{L^{p}(\Omega, \cH)} \lesssim \sum_{j, k = 1}^3 \Big(  \int_{s}^t \int_{\R^{2d}} \| m^{(j)}_{r,z}(u,\bar z_1 )\|_p  \| m^{(k)}_{r,z}(u,\bar z_2 )\|_p 
		\,  R(\bar z_1 - \bar z_2)\, d\bar z\, du\Big)^{\f 1 2}\;.
	\end{equ}
	Given the a priori estimates from Lemma~\ref{lem:moments}, for $\beta \leq \beta_0(3p)$ we have:
	\begin{equs}
		 \| m^{(1)}_{r,z}(u,\bar z )\|_p &\lesssim P_{r- u}(z - \bar z) (1+ r -s)^{\f {2 - \kappa }4} \1_{[s, r]}(u) \;, \\
		 \| m^{(2)}_{r,z}(u,\bar z )\|_p &\lesssim P_{u-s}(\bar z - y) (1+ r -u)^{\f {2 - \kappa }4} \1_{[s, r]}(u) \;,\\
		 \| m^{(3)}_{r,z}(u,\bar z )\|_p &\lesssim  P_{u-s}(\bar z - y) \1_{[r, t]}(u)\;.
	\end{equs}
	Hence, by Lemma~\ref{lem:gaussian_cov_est}, for $j = k =1$ we obtain
	\begin{equs}
		\int_{s}^t &\int_{\R^{2d}} \Bigl(\prod_{i = 1}^2  \| m^{(1)}_{r,z}(u,\bar z_i )\|_p \Bigr)  R(\bar z_1 - \bar z_2)\, d\bar z\, du \\
		&\lesssim  (1 + r -s)^{1 -\f { \kappa }2}  	\int_{s}^r \int_{\R^{2d}} \Bigl(\prod_{i=1}^2 P_{r- u}(z - \bar z_i) \Bigr)  R(\bar z_1 - \bar z_2)\, d\bar z\, du \lesssim (1 + r -s)^{1 -\f { \kappa }2}  \;.
	\end{equs}
	On the other hand, by  \eqref{eq:time_decay_est} and \eqref{e:alpha_beta_decay}, 
	\begin{equs}
		\int_{s}^t &\int_{\R^{2d}}  \| m^{(1)}_{r,z}(u,\bar z_1 )\|_p  \| m^{(2)}_{r,z}(u,\bar z_2 )\|_p  \,  R(\bar z_1 - \bar z_2)\, d\bar z\, du \\
		&\lesssim  (1 + r -s)^{\frac {2 - \kappa} 4}  	\int_{s}^r (1 + r -u)^{\frac {2 - \kappa} 4} \int_{\R^{2d}}  P_{r- u}(z - \bar z_1)  P_{u-s}(\bar z_2-y) \,  R(\bar z_1 - \bar z_2)\, d\bar z\, du \\
		& \lesssim (1 + r -s)^{\frac {2 - \kappa} 4}   \int_{s}^r (1 + r -u)^{\frac {2 - \kappa} 4}  (1 + u-s)^{-\frac \kappa 2} \, du \lesssim  (1 + r -s)^{1 -\f { \kappa}2}  \;.
	\end{equs}
	Estimating in a similar way the other terms in \eqref{e:DM_bound} involving $m^{(2)}_{r,z}$ and $m^{(3)}_{r,z}$, we obtain 
	\begin{equ}
			\|   D M(r, z) \|_{L^{p}(\Omega, \cH)} \lesssim  \beta (1+ r -s)^{\f {2 - \kappa}4}\;.
	\end{equ}
	With this, using again \eqref{eq:time_decay_est} and \eqref{e:alpha_beta_decay}, we return to estimate $DF$:
	\begin{equs}
		&\| DF \|_{L^p(\Omega, \cH \otimes \cH)} \\
		  &\quad \lesssim  
		\beta \Big(
		\int_{s}^t  (1+ r -s)^{1 - \f \kappa 2}\int_{\R^{2d}}  P_{t-r}(x-z_1 ) P_{t-r}(x-z_2) R(z_1-z_2)\,dr\,dz   \, dz\, dr  \Big)^{\f 1 2} \\
		&\quad \lesssim \beta \Big(
		\int_{s}^t   (1+ r -s)^{1 - \f \kappa 2} (1 + t -r)^{- \f \kappa 2} \, dr \Big)^{\f 1 2} \lesssim \beta (1 + t - s)^{\f {2 - \kappa }4}\;.
	\end{equs}
	Gathering bounds back to \eqref{e:E2_norm_bound}, we obtain $\| E_2(s, t, y, x) \|_p \lesssim \beta \, (1 + t-s )^{\frac {2-\kappa} 4}$.
\end{proof}

The following kernel estimate leads to the $L^p$ homogenisation rate.  

\begin{lemma}\label{lem:Kernel_P_bound}
	Let $x\in \R^d$, $t >0$,  $P_r(x)$ the heat kernel, and let
	\begin{equ}
		H_{t, x}(r,z) = P_{r(t-r)/t} (z+ \tfrac r t x) - P_{r}(z) - P_{t-r}(z+x)\;.
	\end{equ}
	Let $\tilde \mu = 1 -\f 2\kappa$, then 
	the following holds:
	\begin{equ}
		J_{t}(x) := \int_0^t \int_{\R^{2d}}	\Bigl(\prod_{i = 1}^2 | H_{t, x}(r,z_i) | \Bigr) R(z_1 - z_2) \, dz \,dr  
		\lesssim (1 + t)^{-\tilde \mu} (1 + (1 + t)^{-1}|x|^2)\;.
	\end{equ}
\end{lemma}

\begin{remark}
Since $P_{r(t-r)/t} (z+ \tfrac r t x) =  q_{0, t-s}^{x,0} (r, z)$, this is really just a 
quantitative way of stating that, when considered over a very large time interval, a Brownian bridge 
is close to a Brownian motion near either end of the interval. Furthermore, the integral $J_t$ is such
that the contribution coming from the ``bulk'' is small.
\end{remark}

\begin{proof}
	Let $\alpha  = \frac 2 \kappa$ such that $\tilde \mu = \alpha (\f \kappa 2 -1)$.
	By Lemma~\ref{lem:gaussian_cov_est}, we have $J_{t, x} \lesssim 1$ uniformly in $t$. Then, in the following we can assume $t > 2^{\f 1 { 1-\alpha}}$. 
	We consider three regions of integration,
	\begin{equ} 
		J_{t}(x)  = J_{0, t^\alpha}(x) + J_{ t^\alpha, t - t^\alpha}(x) + J_{t - t^\alpha, t}(x)\;, 
	\end{equ}
	where 
	\begin{equ}
		J_{t_1, t_2}(x) = \int_{t_1}^{t_2} \int_{\R^{2d}}\Bigl(\prod_{i = 1}^2 | H_{t, x}(r,z_i) | \Bigr) R(z_1 - z_2) \, dz \,dr\;. 
	\end{equ}
	By changing variables $z \mapsto z+x$ and $r \mapsto t -r$, we have $J_{0, t^\alpha}(x)  =  J_{t - t^\alpha, t}(-x)$. Hence, it suffices to estimate $J_{0, t^\alpha}(x)$ and $J_{ t^\alpha, t - t^\alpha}(x)$.

	For $J_{ t^\alpha, t - t^\alpha}(x)$, since $r\in [t^\alpha, t - t^\alpha]$, using \eqref{eq:time_decay_est} for each Gaussian kernel: 
	\begin{equs}
		J_{ t^\alpha, t - t^\alpha}(x) &=  
		\int_{t^\alpha}^{t - t^\alpha} \int_{\R^{d}} | H_{t, x}(r,z_1) | \Big( \int_{ \R^d} |H_{t, x}(r,z_2) |  \,R(z_1 - z_2) \, dz_2 \Big)\, dz_1 \,dr \\
		& \lesssim  \int_{t^\alpha}^{t - t^\alpha} r^{- \f \kappa 2}  \int_{\R^{d}} | H_{t, x}(r,z_1) | \, dz_1 \, dr \lesssim  t^{\alpha(1 - \f \kappa 2)}\;.
	\end{equs}
	
	In the remainder we estimate $J_{0, t^\alpha}(x)$. For this, we split $H_{t,x}$ into two terms:
	\begin{equ}
		|H_{t,x}(r,z)| = H^1_{t,x}(r,z) + H^2_{t,x}(r,z)\;, 
	\end{equ} 
	where 
	\begin{equ}
		H^{(1)}_{t,x}(r,z) = | P_{r(t-r)/t} (z+ \tfrac r t x) - P_{r}(z) | \;, \qquad H^{(2)}_{t,x}(r,z) = P_{t-r}(z+x)\;.
	\end{equ}
	So that, 
	\begin{equ}
		J_{0, t^\alpha}(x) \leq \sum_{i,j = 1}^2 \int_{0}^{t^\alpha} \int_{\R^{2d}} H^{(i)}_{t,x}(r,z_1) H^{(j)}_{t,x}(r,z_2) \,R(z_1 - z_2) \, dz \,dr =  \sum_{i,j = 1}^2 J^{(i,j)}_{0, t^\alpha}(x)\;.
	\end{equ}
	By \eqref{eq:time_decay_est} from Lemma~\ref{lem:gaussian_cov_est}, we have uniformly in $z_1$:
	\begin{equ}
		\int_{\R^d} H^{(2)}_{t,x}(r,z_2) \,R(z_1 - z_2) \, dz_2 \lesssim (1 + t - r)^{-\f \kappa 2}\;.
	\end{equ}
	Hence, recalling $t > 2^{\f 1 { 1-\alpha}}$, for any $i = 1, 2$, we have 
	\begin{equ}
		J^{(i,2)}_{0, t^\alpha}(x) \lesssim \int_{0}^{t^\alpha} (t - r)^{-\f \kappa 2} \int_{\R^d}  H^{(i)}_{t,x}(r,z_1) \, dz_1 \, dr \lesssim (t - t^\alpha)^{1-\f \kappa 2}
		\lesssim t^{1-\f \kappa 2}\;.
	\end{equ}
	We now need to estimate $J^{(1,1)}_{0, t^\alpha}(x)$. We rewrite 
	\begin{equ}
		H^{(1)}_{t,x}(r,z) = \big|P_{r(t-r)/t} (z+ \tfrac r t x) - P_{r}(z)\big| \leq \int_0^1 \big| \d_\theta\, P_{r - \theta \f {r^2}  t} (z + \theta \tfrac r t x) \big| \, d\theta\;.
	\end{equ}
	Letting $v = r^2 / t$ and $w = r x /t$, with $\theta \in [0,1]$ we can compute 
	\begin{equs}
		\big| \d_\theta\,	P_{r - \theta v} (z + \theta w)\big| &=  \Big|  \frac{d  v}{2(r  - \theta v)} +  \frac{v |z + \theta w|^2 }{2(r  - \theta v)^2}  - \frac{\langle z, w\rangle + \theta |w|^2 }{r  - \theta v} \Big| P_{r - \theta v} (z + \theta w) \\
		& \lesssim  \Big( \frac{ v}{r  -  v}  + \frac{v ( |z|^2 + |w|^2) }{(r  -  v)^2} + \frac{ | z| | w| + |w|^2 }{r  -  v} 
		\Big) P_{r - \theta v} (z + \theta w) \;.
	\end{equs}
	Given $r < t^{\alpha}$ and $t > 2^{\f 1 { 1-\alpha}}$, we have 
	\begin{equs}
		\frac{ v}{r  -  v} = \frac { r/t} {1 - r/t} \lesssim t^{\alpha -1}\;, \qquad&|w| \leq t^{\alpha -1} |x|\;,\\
		 \frac{v}{(r  -  v)^2} \leq  \frac{v}{r^2  - 2rv}  \lesssim t^{-1}\;, \qquad & \frac {|w| }{r  -  v}= \frac { r |x|/t } {r -  r^2/t} \lesssim t^{-1} |x|\;.
	\end{equs}
	Therefore, letting $\mu = \alpha (\kappa/2 -1)/2 = \tilde \mu / 2$, we have the following bound
	\begin{equs}
		\big| \d_\theta\, &P_{r - \theta r^2 / t} (z + \theta \tfrac r t x) \big| / P_{r - \theta r^2 / t} (z + \theta \tfrac r t x)  \\ 
		& \lesssim t^{\alpha -1} + t^{-1} \big( |z|^2 + t^{2(\alpha -1)}|x|^2 \big) + t^{-1} \big( |x||z| + t^{\alpha -1}|x|^2 \big)\\
		& \lesssim t^{\alpha -1} + t^{-1} \big( |z|^2+ |x||z| + t^{\alpha -1}|x|^2 \big) \\
		& \lesssim t^{\alpha -1} + ( t^{-1-2\mu} + t^{\alpha -2})|x|^2 + (t^{-1} + t^{-1+2\mu} )|z|^2 \\
		& \lesssim  t^{\alpha -1} + t^{-1-2\mu} |x|^2  + t^{-1+2\mu}|z|^2\;,
	\end{equs}
	where we used the fact that $|x||z| \lesssim t^{-2\mu} |x|^2 + t^{2\mu} |z|^2$ and $\alpha -1 \leq -2\mu$ when $\alpha \leq 2/\kappa$.
	Thus, 
	\begin{equ}\label{eq:H1_bound}
		H^{(1)}_{t,x}(r,z) \lesssim  \big( t^{\alpha -1} + t^{-1-2\mu} |x|^2  + t^{-1+2\mu}|z|^2\big) \int_0^1 P_{r - \theta r^2 / t} (z + \theta \tfrac r t x) \, d\theta\;.
	\end{equ}
	Uniformly in $\theta$, by Lemma~\ref{lem:gaussian_cov_est} and since $\int_{ \R^d} |z|^2 P_r(z+x) dz  \lesssim r + |x|^2$, we have
	\begin{equs}
		\int_{0}^{t^\alpha} &\int_{\R^{2d}} |z_1|^2  P_{r - \theta r^2 / t} (z_1 + \theta \tfrac r t x) \, H^{(1)}_{t,x}(r,z_2)  \, R(z_1 -z_2)\,dr\, dz  \\
		& = \int_{0}^{t^\alpha} \int_{\R^d} |z_1|^2  P_{r - \theta r^2 / t} (z_1 + \theta \tfrac r t x) \Big(  \int_{ \R^d}  H^{(1)}_{t,x}(r,z_2)  R(z_1 -z_2) \,dz_2\Big) \, dz_1 \, dr \\
		& \lesssim \int_{0}^{t^\alpha} \big(( r - \theta \tfrac {r^2}  t) + \theta^2 \tfrac {r^2}  {t^2} |x|^2 \big) \big( (r(t-r)/t)^{-\f \kappa 2} + r^{-\f \kappa 2})\big) \, dr\\
		& \lesssim  \int_{0}^{t^\alpha} r^{1 -\f \kappa 2}  +  r^{2-\f \kappa 2} {t^{-2}} |x|^2  \, dr \lesssim t^{\alpha(2 - \f \kappa 2)}  + t^{\alpha(3 - \f \kappa 2) -2} |x|^2 \; =\; t^{\alpha - 2\mu} + t^{2\alpha - 2 -2\mu} |x|^2\;,
	\end{equs}
	and 
	\begin{equ}
		\int_{0}^{t^\alpha} \int_{\R^{2d}}  P_{r - \theta r^2 / t} (z_1 + \theta \tfrac r t x) H^{(1)}_{t,x}(r,z_2)  \, R(z_1 -z_2)\,dr\, dz \lesssim 1\;, 
	\end{equ}
	With \eqref{eq:H1_bound} and these estimates, we obtain
	\begin{equ}
		J^{(1,1)}_{0, t^\alpha}(x) \lesssim  t^{\alpha -1} + t^{-1-2\mu} |x|^2  + t^{2\alpha - 3} |x|^2 \lesssim t^{\alpha -1} + t^{-1-2\mu} |x|^2\;.
	\end{equ}
	Therefore, since $\max(1-\f \kappa 2, \alpha - 1 )\leq -2\mu$, we obtain
	\begin{equ}
		J_{0, t^\alpha}(x) \lesssim t^{1-\f \kappa 2} + t^{\alpha -1} + t^{-1-2\mu} |x|^2 \lesssim  t^{-2\mu} + t^{-1-2\mu} |x|^2 \;. 
	\end{equ}
	Combining the estimates of $J_{0, t^\alpha}(x)$ and $J_{ t^\alpha, t-t^\alpha}(x)$, we conclude the proof. 
\end{proof}

\section{Proof of the main results}
\label{sec:main}

In order to prove Theorem~\ref{thm:main-theorem-KPZ}, we are going to consider general transformations of $u$ 
(still as in \eqref{eq:rescaled_mSHE}, i.e.\ with initial condition equal to $1$) and prove the following for the 
class of functions $\fr$  below.
\begin{assumption}\label{transformation}
Let $\fr \in \C^2(\R_+;\R)$ be a function satisfying for some $q >0$,    
	\begin{equ}	\label{eq:ass_transf}
		|\fr(z)| + |\fr'(z)| + |\fr''(z)| \lesssim z^{-q} + z^q\;, \qquad \forall z\in \R_+\;.
	\end{equ}
\end{assumption}

\begin{theorem}\label{thm:gen_transf-convergence}
	Let $d\geq 3$, $\xi$ and $\fr$ satisfy Assumptions~\ref{assump-cov-KPZ} and \ref{transformation} respectively and set 
	\begin{equ}
		u_\eps^\Phi(t,x) = 
		\eps^{1-\frac \kappa 2} \bigl( \fr(u_\eps(t,x))-\E [\fr(u_\eps(t,x))]\bigr)\;.
	\end{equ}
	Then, given any 
	$p\geq 1$,	$\alpha <  1 - \f \kappa 2$ and $\sigma <  - 1 - \f {\kappa} 2$,
	 there exists $\beta_0^{\fr}(\alpha, p) >0$ and $\nu_{\fr} \geq 0$ such that, 
	for all $\beta < \beta_0^{\fr}$ and any good coupling of the noise, 
	the random processes  
	$u_\eps^\Phi$ converge in probability to $ \nu_{\fr} \cU^0$ in $ L^p([0, T], \C^{\alpha}(E))$ as $\eps \to 0$, for any $T > 0$ and compact $E \subset \R^d$. Here, $\cU^0$ denotes the solution to the additive stochastic heat equation 
	\begin{equ}
		\partial_t \cU^0 =\ff 12 \Delta \cU^0 + \beta \xi^0\;, \qquad \cU^0(0, \cdot) \equiv 0\;.
	\end{equ}
	The effective variance is given by $\nu_{\fr} = \E [Z \,\fr' (Z)]$,
	where $Z$ denotes a real-valued random variable with the same law as $\vec Z(t,x)$.
\end{theorem}

The proof of this theorem  will be given in Section~\ref{sec:proofMain} below.
Given a test function $g:\R^d\to \R$ and $\fr: \R_+ \to \R$, in the following we write
\begin{equ} 
	Y_t^{\epsilon, g} =\epsilon^{1-\frac \kappa 2}\int_{\R^d} \bigl(\fr(u_\epsilon(t,x))-\E[\fr (u_\epsilon(t,x))]\bigr)g(x) dx.
\end{equ}
In order to prove Theorem~\ref{thm:gen_transf-convergence}, we first derive a divergence representation 
for $Y_t^{\epsilon, g}$ by using the Clark--Ocone formula. 

\begin{lemma} \label{Clark-Ocone}
	Let $\fr \in \C^2(\R_+)$ satisfy  \eqref{eq:ass_transf},  then for $\beta$ sufficiently small, 
	\begin{equ}
		Y_t^{\epsilon, g} =	\epsilon^{1-\frac \kappa 2} \int_{0}^{\frac t {\epsilon^2}}\int_{ \R^d} \int_{ \R^d} \E[  \fr'(u_\epsilon(t,x)) D_{s,y} u_\epsilon(t,x) | \F_s]\; g(x) dx \,\xi(ds,dy).\label{eq:clark-ocone-Y}
	\end{equ}
\end{lemma}

\begin{proof}
	By Lemma~\ref{lem:transform_moments+decorrelation} and Clark--Ocone formula
	\begin{equs}
		Y_t^{\epsilon, g} &= \int_{\R \times \R^d} \E [ D_{s,y} Y_t^{\epsilon, g} | \F_s] \, \xi(ds,dy)\\
		& = 
		\epsilon^{1-\frac \kappa 2} \int_{0}^{\frac t {\epsilon^2}}\int_{ \R^d} \int_{ \R^d}  \E [ D_{s,y}  \fr(u_\epsilon(t,x)) | \F_s] g(x)\, dx \,\xi(ds,dy)\;,
	\end{equs}
and the claim follows from the chain rule.
\end{proof}

The strategy of proof of~\Cref{thm:gen_transf-convergence} is then as follows.
We use the expression for the Malliavin derivative of the
solution combined with the homogenisation result to argue that, for $\tau \gg s$, 
\begin{equs}
	\E [ D_{s,y}  \fr(u(\tau,x)) | \F_s]
	&=  \beta \E [\fr'(u(\tau,x))u(s,y)w_{s,y}(\tau,x) | \F_s]\\
	&\approx \beta \E [\fr'(\vec Z(\tau,x))\cev Z(s,y) u(s,y)P_{\tau-s}(x-y)\vec Z(\tau,x) | \F_s]\;.
\end{equs}
At that point, we note that the process $ \vec Z$ is expected to mix relatively well and that
furthermore for
$|\tau-s| \gg 1$, $\cev Z(s,y)$ is $\F_s$-independent, of expectation $1$, 
and roughly independent of $\vec Z(\tau,x)$ so that one gets
\begin{equs}[e:MalliavinDer]
	\E [ D_{s,y}  \fr(u(\tau,x)) | \F_s] &\approx \beta u(s,y)P_{\tau-s}(x-y) \E [\fr'(\vec Z(\tau,x))\vec Z(\tau,x)] \\
	&\eqdef \beta u(s,y)P_{\tau-s}(x-y)\nu_\fr \;.
\end{equs}
The approximations are proved in the next section. 
 With these in place, in Section~\ref{sec:proof_General_Thm} we can show $ Y_t^{\epsilon, g}$ converges strongly towards $\cU^0$ and conclude the proof of Theorem~\ref{thm:gen_transf-convergence}. In the final section, we consider the KPZ equation with non-flat initial conditions and prove our main result Theorem~\ref{thm:main-theorem-KPZ}.

\subsection{Main convergence result}

The main result of this section is Proposition~\ref{prop:fluctuations_approx} below, 
which shows that we can approximate $Y_t^{\epsilon, g}$ by $\nu_{\fr}$ times the large-scale fluctuations of the linear SHE which,
by the mild formulation, are given by
\begin{equ} \label{eq:X_linear_SHE}
X_t^{\epsilon, g} =	\beta  \epsilon^{1 - \f \kappa 2} \int_0^{\f t {\epsilon^2}} \int_{ \R^{d}} \int_{ \R^{d}} P_{t / \epsilon^2 -s}(x -y)  u(s,y) \, \eps^d g(\eps x) \, dx \, \xi (ds, dy)\;.
\end{equ}
Then, in Proposition~\ref{prop:approx_X_Ueps}, we show that as $\eps \to 0$ these are close to the solution of the stochastic heat equation \eqref{eq:EW_eqn} driven by noise $\xi^\epsilon$, which can be written as
\begin{equ} \label{eq:Ueps}
\cU_t^{\epsilon}(g) :=   \beta \epsilon^{1 - \f \kappa 2} \int_{0}^{\f t {\epsilon^2}}\int_{ \R^d} \int_{ \R^d} P_{t /{\epsilon^2}-s}(x -y) \,\eps^d g(\eps x)\, dx \,\xi(ds,dy)\;.
\end{equ}
Since $\cU^{\epsilon}$ converges to the limit $ \cU^0$, we conclude that $Y_t^{\epsilon, g}$ converges to $ \nu_{\fr} \cU^0$ as in Theorem~\ref{thm:gen_transf-convergence}. We first write
	\begin{equ} [eq:defn_J_Gamma0]
		J_{\kappa, \delta} (f) \eqdef \int_{\R^{4d}} \bigg( \prod_{i=1}^2 | f(x_i)| |x_i - y_i|^{\delta-d} \bigg) |y_1 - y_2|^{-\kappa} \, dy\, dx \;,
	\end{equ}
and we recall the following \cite[Lem.~4.10]{GHL23_arxiv_preprint}.

\begin{lemma} 	\label{lem:J_kappa_delta_bound}
	For any compactly supported bounded function $f: \R^d\to \R$ and any $\delta \in (0,\f \kappa 2)$, 
	one has $J_{\kappa, \delta} (f)  \lesssim \, M^{2d+2\delta-\kappa} \| f\|_\infty ^2$, where  $M$ denotes the diameter of the support of $f$.
	Furthermore, setting $f_x^\lambda (y)=  \lambda^{-d} f(\frac {y -x} \lambda)$,  one has the scaling relation
	\begin{equ}[eq:scaling_J]
		J_{\kappa, \delta}( f_x^\lambda) = \lambda^{2\delta - \kappa}J_{\kappa,  \delta}(f)\;.
	\end{equ}
\end{lemma}

\begin{proposition}\label{prop:fluctuations_approx}
	Let $p \geq 1$. For $\beta$ sufficiently small, for any test function $g \in \Cc(\R^d)$ and $t > 0$, the following holds for any $\delta_0 \in (0,1)$: 
	\begin{equ} [eq:Y_X_Lp_conv]
		\| Y_t^{\epsilon, g} -  
		\nu_{\fr} 
		X_t^{\epsilon, g} \|_p   \lesssim \epsilon^{ \f 1 2\min( 1-\delta_0, 1 - \f 2 \kappa)} \;,
			\end{equ}
	with precise rate given by \eqref{eq:approx_Lp_SHE}. 
\end{proposition}
\begin{proof}
	Recall we denote with $u^{(s)}(t, x)$ the solution of SHE for $t \geq s$ with initial condition $1$ at time $s$. Without loss of generality, assume $\epsilon < t$ and let $\tau =  t/\epsilon^2$. 
	By \Cref{Clark-Ocone} and the expression for the Malliavin derivative \eqref{eq:malliavin_der_g}, 
		\begin{equ}[e:Y_start]
			Y_t^{\epsilon, g} 
				= 	\beta \epsilon^{1-\frac \kappa 2} \int_{0}^{\tau}\int_{ \R^d} \int_{ \R^d} \E[  \fr'(u(\tau,x)) u(s,y) w_{s,y} (\tau,x) | \F_s]\; \eps^d g(\eps x) \,dx \,\xi(ds,dy) \;.
		\end{equ}
		Given the large times asymptotics, the homogenisation of $w_{s,y}(\tau,x)$ and mixing of the stationary process, the proof consists in approximating this fluctuations expression through separation of timescales.
		Namely, given moments control, the contribution coming from $ \tau - s < \epsilon^{-1}$ is negligeable. On the other hand, we 
		let $\gfr(u) = \fr'(u)u$ and for $ \tau - s > \epsilon^{-1}$ we show
		\begin{equation} \nonumber
			\begin{aligned}
				&\E[  \fr'(u(\tau,x)) u(s,y) w_{s,y} (\tau,x) | \F_s] \\
				& \quad \approx \E[  \fr'(\vec Z(\tau,x)  ) u(s,y) w_{s,y} (\tau,x) | \F_s] \\
				& \quad \approx \E[  \fr'(\vec Z(\tau,x)) u(s,y) P_{ \tau - s}(x - y) C_{s, y}( \tau) \vec Z (\tau,x)  | \F_s] \\
				& \quad \approx  \E[   u(s,y) P_{ \tau - s}(x - y) C_{s, y}(\tfrac {s + \tau} 2) \gfr( u^{(\f  {s + \tau} 2)}(\tau,x))  | \F_s] \\
				& \quad =  P_{ \tau - s}(x - y) u(s,y) \E[ \gfr( u^{(\f  {s + \tau} 2)}(\tau,x)) ] \approx P_{ \tau - s}(x - y) u(s,y) \nu_{\fr}\;.
			\end{aligned}
		\end{equation}
		
		Define the indicator functions $\chiU(s) = \1_{[\tau - \epsilon^{-1}, \tau ]}(s)$ and $\chiB(s) = \1_{[0, \tau - \epsilon^{-1}]}(s)$. Then, to show \eqref{eq:Y_X_Lp_conv}, we write 
		the conditional projection in \eqref{e:Y_start} 
		as
		\begin{equs}
			\E[  \fr' (u(\tau,x))& u(s,y) w_{s,y} (\tau,x) | \F_s] - 
			\nu_\fr P_{ \tau - s}(x - y) u(s, y)   =   \sum_{i=1}^5 S_i( \tau,  x , s, y)\;, 
		\end{equs}
		where 	
		\begin{equ}
			S_1 = \Big( \E [ \fr'(u(\tau,x)) u(s,y) w_{s,y} (\tau,x) |\F_s] - \nu_\fr P_{ \tau - s}(x - y) u(s,y) \Big)\,  \chiB(s)
		\end{equ}
denotes the integrand on $[0, \tau - \epsilon^{-1}]$ and, on the remaining interval, we have $4$ terms:
		\begin{equs}
			S_2 &= \E \big[ \big( \fr'(u(\tau,x)) - \fr'(\vec Z(\tau,x) \big) u(s,y) w_{s,y} (\tau,x) |\F_s\big]\,  \chiU(s) \;, \\[0.3em]
			S_3 &= \E \big[  \fr'(\vec Z(\tau,x)) u(s,y) \big( w_{s,y} (\tau,x) - P_{ \tau - s}( x- y)  C_{s,y} (\tau) \vec Z(\tau,x) \big) |\F_s\big]\,  \chiU(s)\;, 
			\\[0.3em]
			S_4 &= P_{ \tau - s}( x - y)  u(s,y)  \, \E \big[ C_{s,y} (\tau)  \gfr(\vec Z(\tau,x))  - C_{s, y}(\tfrac {s + \tau} 2) 
			\gfr( u^{(\f  {s + \tau} 2)}(\tau,x))|\F_s \big]\,  \chiU(s)\;,  \\ [0.3em]
			S_5 &=  P_{ \tau - s}( x - y)  u(s,y)  \big( \E[ C_{s, y}(\tfrac {s + \tau} 2) 
			\gfr( u^{(\f  {s + \tau} 2)}(\tau,x))|\F_s] - \nu_{\fr} \big)\,  \chiU(s)  \;.
		\end{equs}
		With this decomposition, we can express the difference as follows:
		\begin{equ}
			Y_t^{\epsilon, g} -  \nu_{\fr}  X_t^{\epsilon, g}
			= \beta\epsilon^{1 - \f \kappa 2}  \sum_{k =1}^5 \cI_{k, \epsilon}\;,
		\end{equ}
		where
		\begin{equ}
			\cI_{k, \epsilon} := \int_{0}^{\tau}\int_{ \R^d} \int_{ \R^d} S_k(\tau, x, s, y)\, \eps^d g(\eps x)\, dx \,\xi(ds,dy) \;.
		\end{equ}
		Estimating the $L^p$ norms of $S_j$, we show that each term $\cI_{j, \epsilon}$ in the decomposition converges to $0$. Since, by BDG inequality as in \eqref{eq:BDGineq+}, we have for $k=1,\dots, 5$,
		\begin{equ}[e:Ij_bound]
			\| \cI_{k, \epsilon} \|^2_p \leq  \int_{0}^{\tau}\int_{ \R^{4d}} \Bigl(\prod_{i=1}^2 \| S_k( \tau, x_i, s, y_i)\|_p  |g_\eps(x_i)|\Bigr) \, R(y_1 -y_2) \, ds\,dx\,dy\;,
		\end{equ}
		where $g_\eps(x) = \eps^{d} g(\eps x)$.

	By the moment estimates from Lemmas~\ref{lem:moments}--\ref{lem:transform_moments+decorrelation} (in
	particular \eqref{e:bound_w} to bound $w_{s,y}$), we have $\| S_1( \tau, x, s, y)\|_p\lesssim P_{ \tau - s}( x - y)$
	provided that $\beta$ is sufficiently small.  For $\delta \in (0, d)$, we have (cf. \cite[Lem.~2.1]{GHL23_arxiv_preprint})
		\begin{equ} \label{eq:Gaussian_est}
			P_s(x) \lesssim s^{-\f\delta 2} |x|^{\delta - d}\;.
		\end{equ}
		Inserting this bound into \eqref{e:Ij_bound} yields
		\begin{equ}
		\| \cI_{1, \epsilon} \|_p^2 \lesssim \int_{\tau - \epsilon^{-1}}^{ \tau } \int_{\R^{4d}} \Bigl(\prod_{i = 1}^2 P_{ \tau - s}( x_i - y_i) 
			|g_\eps(x_i)|\Bigr) R(y_1 - y_2) \, ds\,dx\,dy \;.
		\end{equ}
		Recall that $t = \eps^2 \tau$. By the change of variables $x_i \mapsto \eps x_i$, $y_i \mapsto \eps y_i$, $s \mapsto \eps ^2 s$, and the fact that $P_{r/\eps^2}(z/\eps) = \eps^d P_{r}(z)$, we have
		\begin{equ}
			\| \cI_{1, \epsilon} \|_p^2 \lesssim  \eps^{\kappa -2}\int_{t- \epsilon}^{ t } \int_{\R^{4d}} \prod_{i = 1}^2 \Bigl(P_{ t - s}( x_i - y_i) 
			 |g(x_i)|  \Bigr) \eps^{-\kappa}R( \tfrac {y_1 - y_2} \eps) \, ds\,dx\,dy \;.
		\end{equ}
		Since $\eps^{-\kappa}R( \tfrac {y_1 - y_2} \eps) \lesssim |y_1 - y_2|^{-\kappa}$, 
		using \eqref{eq:Gaussian_est} with $\delta_0 < 1$, we estimate
		\begin{equs} 
			\| \cI_{1, \epsilon} \|_p^2
			& \lesssim \epsilon^{\kappa -2}  \int_{t- \epsilon}^{t } \int_{\R^{4d}}  \Bigl(\prod_{i = 1}^2 P_{ t - s}( x_i- y_i) \, |g(x_i)|\Bigr) |y_1 - y_2|^{-\kappa} \, ds\,dx\,dy \\
			& \lesssim \epsilon^{\kappa -2}  \int_0^{\epsilon} s^{-\delta_0}  \, ds \; \int_{\R^{4d}}\Bigl(\prod_{i=1}^2 |x_i- y_i|^{\delta_0 -d}\, |g(x_i)|\Bigr) |y_1 - y_2|^{-\kappa} \,dx\,dy  \\
			& \lesssim \epsilon^{\kappa -2 } \epsilon^{1 - \delta_0 }   J_{\kappa,\delta_0}(g) \;,
		\end{equs}
		where $J_{\kappa, \delta}$ is defined as in \eqref{eq:defn_J_Gamma0}. 
		
		For $\cI_{2, \epsilon}$, first apply Cauchy-Schwartz inequality,  then the moment bounds \eqref{e:bound_w} and \eqref{eq:Lp_rate_transform},		we have 
		\begin{equs}
			\| S_2( \tau, x, s, y) \|_p  &\lesssim \|\fr'(u(\tau, x)) - \fr'(\vec Z(\tau, x)\|_{3p} \| w_{s,y} (\tau, x) \|_{3p} \\
			& \lesssim (1 + \tau)^{\f {2 -\kappa}4} P_{ \tau - s}( x - y) \lesssim \epsilon^{\f {\kappa - 2} 4} P_{ \tau - s}( x - y)  \;. 
		\end{equs}
		By a change of variables and \eqref{eq:Gaussian_est} with $\delta_1 < 1$  as above, from  \eqref{e:Ij_bound} we obtain
		\begin{equs}
			\|\cI_{2, \epsilon} \|_p^2 &\lesssim  \epsilon^{\f \kappa 2-1} \, \int_{0}^{\tau - \epsilon ^{-1}}\int_{ \R^{4d}} \Bigl(\prod_{i = 1}^2 P_{ \tau - s}( x_i - y_i) \, |g_\eps(x_i)|\Bigr) R(y_1 - y_2) \, ds\,dx\,dy \\
			& \lesssim   \epsilon^{\kappa-2} \,\epsilon^{\f \kappa 2-1} \, \int_{0}^{t - \epsilon} \int_{\R^{4d}}  \Bigl(\prod_{i = 1}^2 P_{ t - s}( x_i- y_i) \, |g(x_i)|\Bigr) |y_1 - y_2|^{-\kappa} \, ds\,dx\,dy \\
			&\lesssim \epsilon^{\kappa-2} \, \epsilon^{\f \kappa 2-1} \, t^{1-\delta_1}  \,J_{\kappa,\delta_1}(g)\;. \label{eq:est_chain}
		\end{equs}
		
		For the remaining $\cI_{j, \epsilon}$, we shall make use of the fact that $s < \tau - \epsilon^{-1}$. 
		Using the uniform moment bounds on $\gfr(u) = \fr'(u)u$ and on $u$,  and the homogenisation result Proposition~\ref{prop:rho_homogenisation}, we have 
				\begin{equs} 
		\|S_3\|_p &\lesssim  P_{ \tau - s}( x - y) \,  \| \rhosy (\tau, x)\|_{2p} \\
		&\lesssim P_{ \tau - s}( x - y)  \, (1 + \tau -s)^{-\mu} \big(1 + (1 + \tau -s)^{-\f 1 2}|x -y|\big) \\
		&\lesssim 
		\epsilon^{\mu} P_{ \tau - s}( x - y)  +    \epsilon^{\mu} |x -y| P_{ \tau - s}( x - y) \;.
	\end{equs}
	Here  $\mu = \f 12 -\f 1 \kappa$ and  $\rho_{s,y}$ is given by (\ref{e:defn_rho}).
	
	As for $\cI_{2, \epsilon}$, the contribution from $\epsilon^{\mu} P_{ \tau - s}( x - y) $ can be bounded by 
	$\epsilon^{\kappa-2} \, \epsilon^{2\mu} \, t^{1-\delta_1}  \,J_{\kappa,\delta_1}(g)$. 
				The contribution from $  \epsilon^{\mu} |x -y| P_{ \tau - s}( x - y)$ can be estimated as follows:
		\begin{equs}
		\epsilon^{2\mu } &\int_{0}^{\tau - \epsilon ^{-1}} 
		\int_{ \R^{4d}} \Big( \prod_{i = 1}^2  |x_i-y_i| P_{ \tau - s}( x_i - y_i) |g_\eps(x_i)|\Big) R(y_1 - y_2) \, ds\,dx\,dy \\
		& \quad \lesssim \epsilon^{2\mu } \epsilon^{\kappa-2} \int_0^{t -\epsilon} (t-s)^{-\delta_2}  	\, ds \;
		\int_{\R^{4d}} \Bigl(\prod_{i = 1}^2 |x_i- y_i|^{1 + \delta_2 -d}\, |g(x_i)|\Bigr) |y_1 - y_2|^{-\kappa} \,dx\,dy  \\
		& \quad \lesssim \epsilon^{2\mu }  \epsilon^{\kappa-2}  \, t^{ 1 - \delta_2 } \,J_{\kappa, 1 + \delta_2 }(g)\;.
		\end{equs}
		We have employed again the heat kernel bound \eqref{eq:Gaussian_est}, with $\delta_2 < \frac \kappa 2-1$. Summing these up, we have:		\begin{equ}
		\|\cI_{3, \epsilon} \|_p^2 \lesssim  \epsilon^{2\mu} \epsilon^{\kappa-2}  \big( t^{1-\delta_1}  J_{\kappa, \delta_1 }(g) + t^{ 1- \delta_2 } J_{\kappa, 1 + \delta_2 }(g) \big)\;. 
		\end{equ}
	
		For $\cI_{4, \epsilon}$ and $\cI_{5, \epsilon}$, we make use of mixing, separating $C_{s,y}$ and $\vec Z$ in independent parts. We first apply
 \eqref{eq:Lp_rate_transform}. For $s < \tau - \epsilon^{-1}$ and $q > 1$ we have
		\begin{equ}
			\|	\gfr(\vec Z (\tau, x) )- \gfr( u^{(\f  {s + \tau} 2)}(\tau, x))\|_q \lesssim \beta (1 + \tfrac 1 2 ( \tau - s))^{\f {2 -\kappa}4} \lesssim \epsilon^{\f {\kappa-2}4}\;, 
		\end{equ}
		On the other hand, by \eqref{e:estimate_Csy} we also have 
		\begin{equ}
			\|	  C_{s,y} (\tau) - C_{s,y} (\tfrac  {s + \tau} 2)  \|_q \lesssim \beta (1 + \tfrac 1 2 ( \tau - s))^{\f {2 -\kappa}4} \lesssim \epsilon^{\f {\kappa-2}4}\;.
		\end{equ}
		Therefore,  for sufficiently small $\beta$,  
		\begin{equ}
			\|	  C_{s,y} (\tau) \gfr( \vec Z (\tau, x)) - C_{s,y} (\tfrac  {s + t/\epsilon^2} 2) \gfr( u^{(\f  {s + \tau} 2)}(\tau, x)) \|_{2p} \lesssim \epsilon^{\f {\kappa-2}4}\;,
		\end{equ}
		We have used trinagle inequality and \Cref{lem:moments}. Consequently,
		\begin{equ}
			\|S_4\|_p \lesssim \epsilon^{\f {\kappa-2}4} P_{ \tau - s}( x - y)\;.
		\end{equ}
		This is the same estimates as for $\cI_{2, \epsilon}$, therefore,
		 $\| \cI_{4, \epsilon}\|_p^2 \lesssim   \epsilon^{\kappa-2} \epsilon^{\f\kappa 2 -1} t^{1-\delta_1} J_{\kappa, \delta_1 }(g) $.
		 		
	By the integral expression for $C_{s,y}$ in \eqref{e:Csy_eqn},  $C_{s,y} (\tfrac  {s + \tau} 2)$ is of expectation $1$  and independent of $u^{(\f  {s + \tau} 2)}(\tau, x)$. In addition, both terms  are $\F_s$-independent, therefore we have
		\begin{equs}
			\E[ C_{s,y} (\tfrac  {s + \tau} 2)
			\gfr( u^{(\f  {s + \tau} 2)}(\tau, x))|\F_s] &
		 = \E [ C_{s,y} (\tfrac  {s + \tau} 2)  ] \E [ \gfr( u^{(\f  {s + \tau} 2)}(\tau, x)) ]\\
			&  = \E [ \gfr( u^{(\f  {s + \tau} 2)}(\tau, x))  - \gfr( \vec Z (\tau, x)) ]  + \nu_{\fr}\;. 
		\end{equs}
		Using the bounds in \eqref{eq:Lp_rate_transform}, we obtain 
		\begin{equs}
			\|S_5\|_p &\lesssim  P_{ \tau - s}( x - y) \| \gfr( u^{(\f  {s + \tau} 2)}(\tau, x))  - \gfr( \vec Z (\tau, x)) \|_{2p} \\
			&\lesssim P_{ \tau - s}( x - y)  (1 + \tfrac 1 2 (\tau - s))^{\f {2 -\kappa}4} \lesssim \epsilon^{\f {\kappa-2}4} P_{ \tau - s}( x - y) \;. 
		\end{equs}
		This is the same bound for $S_4$, hence $\| \cI_{5, \epsilon}\|_p^2 \lesssim  \epsilon^{\kappa-2} \epsilon^{\f\kappa 2 -1} t^{1-\delta_1}  J_{\kappa, \delta_1 }(g)$. Putting all estimates together,  observing that $\mu= \f 1 2 - \f 1 \kappa < \frac {\kappa -2 } 4$,  we obtain
		\begin{equ} [eq:approx_Lp_SHE]
			\| Y_t^{\epsilon, g} -  \nu_{\fr} X_t^{\epsilon, g} \|_p \lesssim  \eps^{\frac {1-\delta_0} 2}  J_{\kappa, \delta_0 }(g)^{\f 1 2} + \eps^{\mu} 
			\Big(   
			t^{\f {1 - \delta_1} 2} J_{\kappa, \delta_1 }(g)^{\f 1 2} +  t^{\f {1 - \delta_2} 2}
			J_{\kappa, 1+\delta_2 }(g)^{\f 1 2} \Big)\;,
		\end{equ} 
		where $\delta_0, \delta_1 \in (0, 1)$ and $\delta_2 \in (0, \f \kappa 2 -1)$.
		The implicit constants are uniform in $t, \eps > 0$. We conclude the proof of \eqref{eq:Y_X_Lp_conv}. 
	\end{proof}

\begin{proposition}\label{prop:approx_X_Ueps}
Let $p \geq 1$ and $\beta$ sufficiently small. The following holds
for any  $g \in \Cc(\R^d)$,  $t > 0$, and $\theta \in (2, \kappa \wedge (d + 2 - \kappa))$:
\begin{equ} [eq:X_Ueps_Lp_conv]
\|  X_t^{\epsilon, g}  - \cU_t^{\epsilon}(g)  \|_p \lesssim  \epsilon^{\f {\theta -2} {2p}} t^{\f  {1-\delta} 2 (1 + \f 1 p) }   \, J_{\kappa + \theta - 2\delta, \delta}(g)^{\f 1 {2p}} J_{\kappa, \delta }(g)^{\f 1 2 (1 - \f 1 p)}\;.
\end{equ}
\end{proposition}

\begin{proof} 
 Recall that $X^{\epsilon, g}_t $ is defined in \eqref{eq:X_linear_SHE} and $\cU_t^{\epsilon}$ is defined in \eqref{eq:Ueps}.
	 In the former, we substitute $u(s,y)$ by $1+\int_0^s\int_{\R^d}P_{s-r}(y-z)  u(r,z) \xi(dr, dz)$ allowing to recover $ \cU_t^{\epsilon}(g) $ and an additional term: $
	X^{\epsilon, g}_t =    \cU_t^{\epsilon}(g) + B_{\eps}(g)\;$, where
	\begin{equs} \label{eq:side_B_defn}
	B_{\eps} &= \beta^2 \epsilon^{1 - \f \kappa 2}\int_{0}^{\tau}\int_{ \R^d} \int_{0}^{s}\int_{ \R^d} \int_{ \R^d} P_{\tau-s}(x -y) \cdot \\
		&  \hspace{5em} P_{s-r}(y-z) u(r,z) \,\eps^d g(\eps x) dx \,\xi(dr,dz)\,\xi(ds,dy)\;
	\end{equs}
	and $\tau = t / \epsilon^2$.
	We first show that $B_{\eps}$ converges to $0$ in $L^2$. By It\^{o}'s isometry,
	\begin{equs}
		\| B_{\eps} \|_2^2  & = \beta^4 \epsilon^{2 -\kappa }\int_{0}^{\tau}  \int_{0}^{s} \int_{ \R^{6d}}  \Bigl( \prod_{i = 1}^2  P_{\tau-s}(x_i -y_i) P_{s-r}(y_i-z_i) \Bigr)  R(y_1 - y_2)  \\
		& \qquad  \qquad  R(z_1 -z_2)  \E [u(r,z_1)u(r,z_2) ] \, \eps^{2d}g(\eps x_1) g(\eps x_2) \,ds\, dr\, dx\, dy\, dz \;.
	\end{equs}
	Using the uniform $L^2$ bound on $u$ and a change of variables as before\footnote{Namely, 
		$x_i \mapsto \eps x_i$, $y_i \mapsto \eps y_i$, $z_i \mapsto \eps z_i$, $s \mapsto \eps ^2 s$, $r \mapsto \eps ^2 r$, combined with $P_{r/\eps^2}(z/\eps) = \eps^d P_{r}(z)$.}, we obtain 
		\begin{equs}
			\| B_{\eps} \|_2^2  & \lesssim   \int_{0}^{t} \int_{0}^{s}   \int_{\R^{6d}} \Bigl( \prod_{i = 1}^2  P_{t -s}(x_i -y_i) P_{s-r}(y_i-z_i) \Bigr) \, \epsilon^{-2}R( \tfrac {y_1 - y_2} \epsilon) \\
			& \hspace{5em}\epsilon^{-\kappa}R( \tfrac {z_1 - z_2}\epsilon) |g(x_1) g(x_2) | \,ds\, dr\, dx\, dy\, dz \;.
		\end{equs}
		Using $\epsilon^{-\kappa}R(z/\epsilon) \lesssim |z|^{-\kappa} $ and the kernel bound \eqref{eq:Gaussian_est} with $\delta \in (0, 1)$, we obtain
		\begin{equs}
			\| B_{\eps} \|_2^2  
			 \lesssim \int_{0}^{t} \int_{0}^{s} (t-s)^{-\delta} (s-r)^{-\delta}  \,dr\,ds  \, I^\eps_{\delta, \kappa}(g)  \lesssim t^{2 - 2\delta}   \, I^\eps_{\delta, \kappa}(g) \;, 
		\end{equs}
		where, with $G(x) = g(x_1) g(x_2)$, 
		\begin{equ}
			I^\eps_{\delta, \kappa}(g) = \int_{ \R^{6d}}  \Bigl( \prod_{i = 1}^2  |x_i -y_i|^{\delta - d} |y_i-z_i|^{\delta - d} \Bigr) \eps^{-2}R( \tfrac {y_1 - y_2} \epsilon) |z_1 - z_2|^{-\kappa}  |G(x)|\, dx\, dy\, dz\;.
		\end{equ}
		Using the fact that, for $\beta_1 + \beta_2+d<0$ and $\beta_i<-d$,  $\int_{\R^d} |y -z|^{\beta_1} |z-w|^{\beta_2} dz = c_{\beta_1, \beta_2} |y-w|^{\beta_1 + \beta_2 +d}$, we can integrate out the $z_i$ variables and obtain
		\begin{equ}
			I^\eps_{\delta, \kappa}(g)  = c  \int_{ \R^{4d}}  \Bigl( \prod_{i = 1}^2  |x_i -y_i|^{\delta - d} \Bigr) \eps^{-2}R( \tfrac {y_1 - y_2} \epsilon) |y_1 -y_2|^{2\delta - \kappa}  |G(x)|\, dx\, dy\;.
		\end{equ}
		Let us denote
		\begin{equ}
			K_\eps(x) = \eps^{-2}R( x/ \epsilon) |x|^{2\delta - \kappa}\;.
		\end{equ} 
		We know that $R(x) \lesssim |x|^{-\theta}$, for any $\theta \in [0, \kappa]$. Let us pick $\theta \in (2, \kappa \wedge (d + 2 - \kappa))$, so that $\theta + \kappa - d < 2$ and we fix 
		$\delta \in ( \f 1 2 (\theta + \kappa - d ), 1 )$. Then, uniformly in $\eps \in (0, 1)$, 
		\begin{equ}[e:Keps_bound]
			K_\eps(x) \lesssim \eps^{\theta - 2} |x|^{2\delta - \kappa - \theta}\;. 
		\end{equ}
		Note that, by the choice of $\theta$ and $\delta$, $K_\eps(x)$ is integrable at $0$ and decays fast enough at infinity to integrate it over the other kernels. In fact, by \eqref{e:Keps_bound}  we obtain
		\begin{equs}
			I^\eps_{\delta, \kappa}(g) &\lesssim \eps^{\theta - 2}  \int_{ \R^{4d}}  \Bigl( \prod_{i = 1}^2  |x_i -y_i|^{\delta - d} \Bigr) |y_1 -y_2|^{2\delta - \kappa - \theta}  |G(x)|\, dx\, dy\\
			& = \eps^{\theta - 2}  J_{\kappa + \theta - 2\delta, \delta}(g)\;, 
		\end{equs}
		where $J$ is defined as in \eqref{eq:defn_J_Gamma0} and bounded by Lemma~\ref{lem:J_kappa_delta_bound}, as long as $\delta < \frac 1 2 (\kappa + \theta - 2\delta)$. This is the case, given that $\delta< 1$ (thus $\delta < \frac 1 4 (\kappa + \theta)$ as required). Then,

	\begin{equ}
		\| B_{\eps}\|_2 \lesssim \eps^{\f {\theta - 2} 2} t^{1-\delta} J_{\kappa + \theta - 2\delta, \delta}(g)^{\f 1 2}\;.  
	\end{equ} 
	
	On the other hand, estimating the $L^p$ norms of $X^\eps$ and $\cU^\eps$ by BDG and a change of variables as in \eqref{eq:est_chain} in the previous proof, we have  
	\begin{equs}
		\| B_{\eps} \|_p  & \lesssim \| X^{\epsilon, g}_t\|_p  + 	\| \cU^{\epsilon}_t(g)\|_p \\
		& \lesssim   \epsilon^{1 - \f \kappa 2}  \Big( \int_0^{\tau} \int_{ \R^{4d}} \Big(\prod_{i = 1} P_{\tau -s}(x_i -y_i)  \eps^d g(\eps x_i) \Big) R(y_1 - y_2) \, dx \, dy\, ds \Big)^{\f 1 2}\\
		&  \lesssim  \Big( \int^{t}_{0} \int_{\R^{4d}} \Bigl( \prod_{i = 1}^2 P_{ t - s}( x_i- y_i) | g(x_i)|\Bigr) |y_1 - y_2|^{-\kappa} \, ds\,dx\,dy \Big)^{\f 1 2} \\
		& \lesssim  t^{\f {1 - \delta} 2}  J_{\kappa, \delta }(g)^{\f 1 2}\;, 
	\end{equs}
	By interpolation,   for any $\theta \in (2, \kappa \wedge (d + 2 - \kappa))$ and $\delta \in ( \f 1 2 (\theta + \kappa - d ), 1 )$, 
	\begin{equs} [eq:approx_Lp_X_Ueps]
		\|  X^{\epsilon, g}_t - \cU^{\epsilon}_t(g) \|_p &=  	\| B_{\eps} \|_p \leq \| B_{\eps} \|_{2}^{\f 1 p} \,  \| B_{\eps} \|_{2(p-1)}^{1- \f 1 p} \\
		& \lesssim \epsilon^{\f {\theta -2} {2p}} t^{\f  {1-\delta} 2 (1 + \f 1 p) }   \, J_{\kappa + \theta - 2\delta, \delta}(g)^{\f 1 {2p}} J_{\kappa, \delta }(g)^{\f 1 2 (1 - \f 1 p)}\;,
	\end{equs}
	with implicit constants uniform in $t, \eps > 0$, concluding the proof. 
\end{proof}

\subsection{Convergence in Hölder spaces}\label{sec:proof_General_Thm}

Before going into the proof of Theorem~\ref{thm:gen_transf-convergence}, we introduce the spaces of locally H\"{o}lder continuous distributions $\C^\alpha(\R^d)$, of negative exponent, and prove moment estimates in these spaces.   For a function $g \in \C(\R^{d})$, $z \in \R^d$ and $\lambda>0$, we denote
\begin{equ}
	g_x^\lambda(y) :=  \lambda^{-d}g( \lambda^{-1}(y-x))\;,
\end{equ}
where we drop the corresponding index if $\lambda =1$ or $x= 0$. 
We denote with $\D$  the space of smooth functions with compact support, and $\D'$ the dual space
of distributions. 
Given $\alpha < 0$, let $r =  \lceil -\alpha \rceil$ be the smallest positive integer bigger than $-\alpha$ and consider the following subset of $\C^r(\R^d)$:
\begin{equ}
	\B_r := \{ g \in \D: \supp(g)\subset B_1, \|g\|_{\C^r}\le 1 \}\;,  
\end{equ}
where $B_a$ denotes the open ball of radius $a$ centred at $0$. 

We say that a distribution $\zeta$ belongs to $\C^\alpha(\R^d)$ 
if for every compact set $E \subset \R^d$, 
\begin{equ}\label{eq:seminorm_Calpha_basic} 
	\| \zeta \|_{\alpha; E} \eqdef \sup_{g \in \B_r } \sup_{ \lambda \in (0,1)} \sup_{x \in  E}  \,\, \lambda^{-\alpha}| \langle \zeta, g^\lambda_x \rangle | < \infty\;.
\end{equ}
See  \cite[Defn.~3.7]{hairer_RegStr}.
The space $\C^\alpha$ is metrisable, with \begin{equ}\label{eq:metric_Calpha}
	d_{\alpha} (\zeta, \zeta')\eqdef \sum_{m \geq 1} 2^{-m} \big(1 \, \wedge \, \| \zeta - \zeta' \|_{\alpha; B_m} \big)\;,
\end{equ}
making it a Fréchet space. For $\alpha' > \alpha$, $\C^{\alpha'}$ is compactly embedded in $\C^\alpha$ .

The spaces of $\C^\alpha$ 
can be characterised in multiple ways, 
here we use the characterisation from \cite{Cara-Zam} in terms of a single test function. We fix
an arbitrary reference test function $\varphi \in \D$ such that $\supp \varphi \subset B_1$
and $\int \varphi = 1$ for the remainder of the article.  We shall use the following notation for its rescaled version:
\begin{equ}
	\varphi^{(n)} = 2^{nd} \varphi (2^n \cdot)\;,\qquad
	\varphi_x^{(n)} = 2^{nd} \varphi \bigl(2^n (\cdot-x)\bigr)\;.
\end{equ}
Although $\varphi$ is fixed, the following shows that control on $\langle\zeta,\varphi^{(k)}_x\rangle$ allows to get a similar bound for all test functions and control \eqref{eq:seminorm_Calpha_basic}. 

\begin{theorem}[{\cite[Theorem 12.4]{Cara-Zam}}]\label{thm:critHolder}
	Let $\zeta \in \D'$ be a distribution and let $\alpha \in (-\infty, 0]$.
	Then, one has $\zeta \in \C^\alpha$ if and only if, for any compact set $E$ and uniformly over $k\in \N$,
	\[
	\sup_{x\in E} |\langle\zeta,\varphi^{(k)}_x\rangle| \lesssim2^{-k \alpha}.
	\]
	Furthermore, the semi-norms of $\zeta$ can be estimated by
	\begin{equ}[e:boundSeminorm]
		\|\zeta\|_{\alpha;E} \leq C \sup_{x \in E_2}\sup_{ k\in \N} 
		\; {2^{k\alpha}|\<\zeta,\varphi^{(k)}_x\>| }
	\end{equ}
	where $C = C(\varphi, \alpha, d)$ is an explicit constant and $E_2=\{z: d(z,E)\le 2\}$.
\end{theorem}

This characterisation of $\C^\alpha$ spaces allows for a criterion to estimate norms in H\"{o}lder spaces. We recall the following definition from \cite[Defn.~4.6]{GHL23_arxiv_preprint}.

\begin{definition}\label{dis-class}
	A random distribution $\zeta$ is said to belong to $\C^{\alpha_0}_p$ if there exists $C>0$ such that the following estimates hold  for any  $n\geq 1$:
	\begin{equs}
		\label{eq:lattice+time_control}
		\sup_{x \in  \R^d} 
		\bigl\|\langle \zeta, \varphi^{(n)}_x \rangle\bigr\|_{p} &\leq C 2^{-n\alpha_0} , \\
		\label{eq:spatial_control}
		\sup_{\substack{x, y \in  \R^d\;,\\
				0<|x-y|\leq 2^{-n}}}  
		\bigl\|\langle \zeta, \varphi^{(n)}_x - \varphi^{(n)}_y \rangle \bigr\|_{p} &\leq C 2^{-n(\alpha_0-1)} |x-y|\;.
	\end{equs}
	The smallest such constant $C$ is denoted by $\|\zeta\|_{\C^{\alpha_0}_p}$. 
	We write $\C^{\gamma_0, \alpha_0}_{p, T}$, as a shorthand, for the space $\C^{\gamma_0}([0, T], \C^{\alpha_0}_p)$.
\end{definition}

Given any compact $E \subset \R^d$, we denote by $\C^\alpha (E)$ the space of distributions such that the seminorm in \eqref{eq:seminorm_Calpha_basic}  is finite. Then, by \cite[Prop.~4.8]{GHL23_arxiv_preprint} we have the following:

\begin{proposition}\label{prop:time-holder_distribution_process}
	Let $\alpha_0<0$, $p > d$, and $\gamma_0 > \frac 1 {p}$. 
	Then, for any $\gamma \in (0,\gamma_0  - \frac 1 p)$, $\alpha < \alpha_0  - \frac d p$ and any compact $E \subset \R^d$, 
	one has the continuous embeddings
	\begin{equ}
		\C^{\alpha_0}_p \subset L^p(\Omega,\C^{\alpha}(E))\;,\qquad
		\C^{\gamma_0,\alpha_0}_{p,T} \subset L^p(\Omega,\C^\gamma([0, T],\C^{\alpha}(E)))\;.
	\end{equ}
\end{proposition}

%

\begin{remark}
In \cite[Prop.~4.8]{GHL23_arxiv_preprint}, these inclusions are stated for  $E=\R^d$, but the proof 
given there allows to replace $\R^d$ by $E$. 
 \end{remark}

We have the following:

\begin{lemma}\label{lem:unif_Calpha_convergence}
	Given any $\alpha_0 <  1-\f \kappa 2 $, $p \geq 1$, for sufficiently small $\beta$, there exist $a = a(\alpha_0, p, \kappa) > 0$, such that the following holds for any $T >0$:
	\begin{equ}[eq:Yeps_cUeps_approx]
		\sup_{0<t\leq T}	\| Y^\epsilon_t-  \nu_{\fr} \cU^\eps_t \|_{\C^{\alpha_0}_p}   \lesssim  \epsilon^a\;.
	\end{equ}
	In particular, for every $\alpha<\alpha_0-\f d p$, there exists $a > 0$ such that
	 $\|\sup_{0<t\leq T}   \| Y_t^\epsilon - \nu_{\fr} \cU_t^\eps \|_{\alpha, E}\|_p \lesssim  \epsilon^a$, for any compact $E \subset \R^d$. 
	Furthermore, for any $\gamma_0 \in (0,\frac 1 2)$, $\alpha_0 < 1 - \frac\kappa2 - 2\gamma_0$, and $p \ge 1$, one has
$		\sup_{\eps \in (0, 1]} \| \cU^\eps  \|_{\C^{\gamma_0, \alpha_0}_{p, T}}  < \infty$.
\end{lemma}
\begin{proof}
By \eqref{eq:approx_Lp_SHE}, for $\delta_0, \delta_1 \in (0, 1)$, $\delta_2 \in (0, \f \kappa 2 -1)$, and $\mu = \f 1 2 - \f 1 \kappa$, one has
	\begin{equ}
	\sup_{t\in [0,T]}	\| Y_t^{\epsilon, g} -  \nu_{\fr} X_t^{\epsilon, g} \|_p \lesssim 
		\eps^{\frac {1-\delta_0} 2}   J_{\kappa, \delta_0 }( g)^{\f 1 2} + \eps^{\mu} 
		\Big(   J_{\kappa, \delta_1 }( g)^{\f 1 2} +
		J_{\kappa, 1+\delta_2 }(g)^{\f 1 2} \Big).
	\end{equ}
Recall, Lemma~\ref{lem:J_kappa_delta_bound},  $J_{\kappa, \delta}(g) \lesssim M^{2d+2\delta-\kappa} \| g \|^2_\infty$, where $M$ is the diameter of the support of $g$, and the scaling property \eqref{eq:scaling_J}. 

Let $g \in \bar \D$ and $\delta \in (0,1)$. Then  the following holds uniformly in $\lambda \in (0, 1)$:
	\begin{equs}
		\sup_{t>0} \sup_{x\in \R^d}\| Y_t^{\epsilon, g_x^\lambda} -  \nu_{\fr} X_t^{\epsilon, g_x^\lambda} \|_p &\lesssim \eps^{(\f  {1 - \delta} 2 ) \wedge \mu}  \lambda^{\delta - \f \kappa 2 } \big(  J_{\kappa, \delta }(g)^{\f 1 2}  + J_{\kappa, 1+\delta_2 }(g)^{\f 1 2} \big)\\
		&\lesssim  \eps^{(\f  {1 - \delta} 2 ) \wedge \mu}  \lambda^{\delta - \f \kappa 2 }  
		\| g\|_\infty\;.  
	\end{equs}
	Hence, for any $\alpha_0 < 1 - \f \kappa 2$, by picking\footnote{If $\alpha_0 > - \f \kappa 2$, can set $\delta =  \alpha_0 + \f \kappa 2$.} $\delta = \delta(\alpha_0)$ sufficiently close to $1$ and applying this to the test function $\varphi^{(n)}_x$, we have
	\begin{equ}
		\sup_{t \in (0, T] } \sup_{n\geq 1}  \sup_{x \in \R^d}  2^{n\alpha_0}  
		\bigl\|\langle Y_t^{\epsilon} - \nu_{\fr} X_t^{\epsilon}, \varphi^{(n)}_x \rangle\bigr\|_{p} \lesssim  
		\eps^{(\f  {1 - \delta} 2 ) \wedge \mu}\;,
	\end{equ} 
	which implies \eqref{eq:lattice+time_control}. Since the bound holds for any $g \in \bar \D$, \eqref{eq:spatial_control} is also satisfied (cf. \cite[Rem.~4.7]{GHL23_arxiv_preprint}). We obtain
	\begin{equ}
		\sup_{t \in (0, T] } \| Y_t^{\epsilon} - \nu_{\fr} X_t^{\epsilon}\|_{\C^{\alpha_0}_p} \lesssim \eps^{(\f  {1 - \delta} 2 ) \wedge \mu}\;.
	\end{equ} 
	
	
	It remains to control 
	$\|  X_t^{\epsilon} - \cU^\eps_t \|_{\C^{\alpha_0}_p} $.
	Following the same steps as above,  
	for any $\theta \in (2, \kappa \wedge (d + 2 - \kappa))$ and $\delta \in ( \f 1 2 (\theta + \kappa - d ), 1 )$, by Proposition~\ref{prop:approx_X_Ueps} and the scaling relation \eqref{eq:scaling_J},  
	we have uniformly in  $\lambda \in (0, 1)$:
	
	\begin{equs}
	\sup_{x\in \R^d}\sup_{t\in (0,T]}	\| X^{\epsilon, g_x^\lambda}_t - \cU^{\epsilon}_t(g_x^\lambda) \|_p &\lesssim \epsilon^{\f {\theta -2} {2p}} T^{\f  {1-\delta} 2 (1 + \f 1 p)}  \lambda^{ (4\delta - \kappa - \theta) \f 1 {2p}} \lambda^{(2\delta -\kappa) \f {p-1}{2p} } \| g\|_\infty \\
		&\lesssim \epsilon^{\f {\theta -2} {2p}}   \lambda^{(\delta - \f \kappa 2)(1 + \f 1 p)} \| g\|_\infty\;, 
		\end{equs}
	where we used the fact that $\theta < \kappa$. 
	
	For any $\alpha_0 < 1 - \f \kappa 2$, picking sufficiently high $p(\alpha_0) \geq 1$ and $\delta(\alpha_0) < 1$, we have  $(\delta - \f \kappa 2)(1 + \f 1 p) > \alpha_0$. 
	Then, by the same reasoning as above, we obtain
	\begin{equ}
		\sup_{t \in (0, T] }\|  X_t^{\epsilon} - \cU^\eps_t \|_{\C^{\alpha_0}_p} \lesssim \epsilon^{\f {\theta -2} {2p}} \;.
	\end{equ}
	Letting $a = \f 1 2 \min( 1 - \delta, 1 - \f 2 \kappa, \f {\theta -2} {p}  )$, by triangle inequality we 
	conclude that \eqref{eq:Yeps_cUeps_approx} holds.
	The final claim follows  by Proposition~\ref{prop:time-holder_distribution_process}, as long as  $\alpha< \alpha_0 - \f d p$. 
	
	Regarding the bound on $\cU_{t}^\eps$, we note that the definition of $\xi^\eps$ implies
	 the scaling relation
	\begin{equ}
	\cU_t^\eps(g_x^\lambda) = \eps^{1-\f\kappa2}\cU_{t/\eps^2}(g_{x/\eps}^{\lambda/\eps})\;,
	\end{equ}
so that the required bound follows immediately from \cite[Lem.~4.14]{GHL23_arxiv_preprint}.
	\end{proof}

%


\subsection{Proof of the main results}
\label{sec:proofMain}


\begin{lemma}\label{lem:convergenceIndicator}
Consider a  good coupling for the $\xi^\eps$ satisfying Assumption~\ref{assump-cov-KPZ}.
Given a test function $\psi \in \C_c^\infty(\R^{d+1})$ and $t \in \R$, write
$\tilde \psi_t(s,x) = \psi(s,x)\mathbf{1}_{s \ge t}$. Then,
one has
$\lim_{\eps \to 0} \xi^{\eps}(\tilde \psi_t) = \xi^0(\tilde \psi_t)$ in probability.
\end{lemma}

\begin{proof}
Let $\chi\colon \R \to \R_+$ be a smooth increasing function such that 
$\chi(s) = 0$ for $s \le 0$ and $\chi(s) = 1$ for $s \ge 1$, and set $\chi_\delta(t) = \chi(t/\delta)$.
Setting $\tilde \psi_t^\delta(s,x) = \chi_\delta(s-t)\psi(s,x)$, we then note that, as a consequence of 
Itô's isometry, one can find a constant $C$ (depending on $\psi$) such that
\begin{equ}[e:ItoBound]
\E \xi^\eps(\tilde \psi_t^\delta - \tilde \psi_t)^2 \le C\delta\;,
\end{equ}
uniformly over all $\eps,\delta \le 1$. Since $\tilde \psi_t^\delta$ is a smooth function,
one furthermore has $\xi^\eps(\tilde \psi_t^\delta) \to \xi^0(\tilde \psi_t^\delta)$ in probability
as $\eps \to 0$ for any fixed $\delta$. Combining this with \eqref{e:ItoBound}, the claim follows
at once.
\end{proof}

With the Lemmas above we are now in the position to prove Theorem~\ref{thm:gen_transf-convergence}.  Write $L^p([0, T], \C^\alpha(E))$ for the space with seminorms, \begin{equ}
	\| \zeta \|_{\alpha, E, p, T} := \Big( \int_0^T \| \zeta_t \|_{\alpha, E}^p\, dt \Big)^{\f 1 p}\;,
\end{equ}
where the semi-norms of $\C^\alpha(E) $ is as in Sec. \ref{sec:proof_General_Thm}.
We see that for any $\gamma >0$, the space $\C^\gamma([0, T], \C^{\alpha}(E) )$ is continuously embedded in $L^p([0, T], \C^\alpha(E))$.

\begin{proof}[Proof of Theorem~\ref{thm:gen_transf-convergence}] 
From Lemma~\ref{lem:unif_Calpha_convergence}, given $\alpha < 1 - \f \kappa 2$, we pick  $p = p(\alpha)$ suffiently high such that, as $\eps \to 0$, 
\begin{equ}
 \Bigl\| 	\sup_{t \in [0, T]} \| Y_t^\epsilon - \nu_{\fr} \cU_t^\eps \|_{\alpha, E}\Bigr\|_p
	\longrightarrow 0\;. 
\end{equ}
This implies that $Y^\epsilon -  \nu_{\fr} \cU^\eps$ converges in probability to zero in $L^p([0, T], \C^{\alpha}(E))$. 

It therefore remains to show that, for any good coupling, one has 
$\cU^\eps \Rightarrow \cU^0$ in probability in $L^p([0, T], \C^{\alpha}(E))$.
For this, we make use of \cite[Prop.~3.12]{GenDiff} which implies that it suffices to show tightness of the laws of $\cU^\eps$
and convergence in probability of $\cU^\eps(\psi)$ for any test function $\psi \in \C_c^\infty(\R^{d+1})$.

By Lemma~\ref{lem:unif_Calpha_convergence} and Proposition~\ref{prop:time-holder_distribution_process}, we obtain for any $\gamma' \in (0, \f 1 2)$ and any $\alpha' < 1 - \f \kappa 2 - 2\gamma'$, 
\begin{equ}
	\sup_{\eps \in (0, 1]}  \norm{ \| \cU^\eps \|_{\C^{\gamma'}([0, T], \C^{\alpha'}(E) )} }_p \lesssim 1\;.
\end{equ}
Since $\C^{\gamma'}([0, T], \C^{\alpha'}(E) )$ is compactly embedded in $\C^{\gamma}([0, T], \C^{\alpha}(E) )$ for any $\gamma \in (0, \gamma')$ and $\alpha < \alpha'$, we conclude that the required tightness holds.

To get convergence of $\cU^\eps(\psi)$, we note that $ \cU^\eps(\psi) = \xi^\epsilon(P\psi)$ with
\begin{equ}
(P\psi)(s,y) = \beta \mathbf{1}_{s \ge 0}\int_0^\infty \int_{\R^d} \psi(t,x) P_{t - s}(x - y) \,dx\,dt\;,
\end{equ}
so that the desired convergence follows from Lemma~\ref{lem:convergenceIndicator}.
\end{proof}

We now have all the ingredients in place to prove the main result of this article, namely
the identification of the large-scale limit of the KPZ equation given in Theorem~\ref{thm:main-theorem-KPZ}.

\begin{proof}[Proof of Theorem~\ref{thm:main-theorem-KPZ}]
	Given the initial condition $h_{0}$, let $${\psi_\epsilon(y) = \exp(\epsilon^{\frac \kappa 2 -1} h_{0}(y))  -1}$$ and denote by 
$w^\epsilon_{0,z}(t,x) = \eps^{-d} w_{0,z/\eps}( t /\epsilon^2,x/\epsilon) $ the rescaled fundamental solution of the linear stochastic heat equation. Since the initial condition for $u_\epsilon$ is $1$, it follows from the Cole--Hopf transform that
	\begin{equ}
		h_\eps(t,x) = \eps^{1-\f\kappa2} \log \Bigl( u_\eps (t,x) + \int_{\R^d} w^\epsilon_{0,y}(t,x) \psi_\epsilon(y) \, dy \Bigr)\;.
	\end{equ}
	Since $2\abs{\log(x+y) - \log(x) - y/x} \le (y/x)^2$ for $x,y > 0$, it follows that
	\begin{equs}
	\Bigl| h_\eps(t,x) - \eps^{1-\f \kappa 2}& \log  \big( u_\eps (t,x) \big) - \eps^{1-\f\kappa2} (u_\eps (t,x))^{-1} \int_{\R^d} w^\epsilon_{0,y}(t,x) \psi_\epsilon(y) \, dy  \Bigr| \\
	  &	\le  \f 12 \eps^{1-\f\kappa2} \Big( (u_\eps (t,x))^{-1} \int_{\R^d} w^\epsilon_{0,y}(t,x) \psi_\epsilon(y) \, dy \Big)^2\;.
	\end{equs}
	By the a priori bounds for $u_\eps (t,x))^{-1}$ from Lemma~\ref{lem:moments} and (\ref{e:bound_w}) and since
	$\|\psi_\eps\|_\infty \lesssim \eps^{\f\kappa2-1}$, we obtain for any $p \geq 1$ 
	\begin{equ}
		\Bigl\| (u_\eps (t,x))^{-1} \int_{\R^d} w^\epsilon_{0,y}(t,x)\psi_\eps(y)\, dy \Bigr\|_p \lesssim \eps^{\f\kappa2-1}\int_{\R^d} P_{t} (x-y)\, dy = \eps^{\f\kappa2-1}\;.
	\end{equ}
	On the other hand, by Taylor approximating $\psi_\epsilon(y)$, we have
	$		\| \psi_\epsilon(y) - \epsilon^{\frac \kappa 2 -1}  h_{0} \|_\infty  
	\lesssim \epsilon^{\kappa -2}
	$ and consequently,
	$$|\eps^{1-\f\kappa2}  (u_\eps (t,x))^{-1} \Bigl(\int_{\R^d} w^\epsilon_{0,y}(t,x) \psi_\epsilon(y) \, dy -\int_{\R^d} w^\epsilon_{0,y}(t,x) h_0(y) \, dy\Bigr) |\lesssim \epsilon^{\f \kappa 2-1},$$
	 Combining these bounds, we conclude that 
	\begin{equ}
		\Big\| h_\eps(t,x) - \eps^{1-\f \kappa 2} \log  \big( u_\eps (t,x) \big) - (u_\eps (t,x))^{-1} \int_{\R^d} w^\epsilon_{0,y}(t,x) h_0(y) \, dy  \Bigr\|_p 
		\lesssim  \eps^{\f\kappa2-1}\;.
	\end{equ}
	
	As a consequence of Theorem~\ref{theo:homogenisation}, with $ \mu = \f 1 2 - \f 1 \kappa$, for any $p\geq1$ we have 
	\begin{equs}
		&\|	w^\epsilon_{0,y}(t,x) -  \vec Z(t/\epsilon^2, x/\epsilon) \cev Z (0, y/\epsilon) \eps^{-d}P_{t/\eps^2} (\tfrac {x-y}\eps)\|_p \\
		& \qquad  \lesssim (t/\eps^2)^{-\mu} (1 + (t/\eps^2)^{-\f 1 2} \eps^{-1}|x-y|) \eps^{-d}P_{t/\eps^2} (\tfrac {x-y}\eps)  \\ &\qquad \lesssim \eps^{2\mu} t^{-\mu}  (1 + t^{-\f 1 2} |x-y| ) P_{t}(x-y)\;,
	\end{equs}
	where we used the fact that $\eps^{-d}P_{t/\eps^2} (\tfrac x\eps) = P_t(x)$. 
	
	Given that $h_0$ is bounded and the uniform control of negative moments of $u_\eps$ from Lemma~\ref{lem:moments}\ref{i}, we obtain
	\begin{equs}
		&	\Big\| \int_{\R^d}  (u_\eps (t,x))^{-1} \Big( w^\epsilon_{0,y}(t,x) - \vec Z(t/\epsilon^2, x/\epsilon) \cev Z (0, y/\epsilon) \eps^{-d}P_{t/\eps^2} (\tfrac {x-y}\eps) \Big) h_0(y) \, dy   \Big\|_p \\
		& \qquad \qquad \lesssim  \eps^{2\mu} t^{-\mu} \,  \| u^{-1}_\eps (t,x) \|_{2p} 	\int_{ \R^{d}} (1 + t^{-\f 1 2} |x-y| ) P_{t}(x-y) \, dy \lesssim \eps^{2\mu} t^{-\mu} \;.
	\end{equs}
	On the other hand, by Theorem~\ref{thm:stat_field} we also have 
	\begin{equs}
		\Big\| \frac {\vec Z(t/\epsilon^2, x/\epsilon)}{u_\eps (t,x)} - 1  \Big\|_p \lesssim \| u^{-1}_\eps (t,x) \|_{2p} \| \vec Z(t/\epsilon^2, x/\epsilon) - u_\eps (t,x) \|_{2p} \lesssim t^{\f {2-\kappa} 4} \eps^{\f \kappa 2 -1}\;. 
	\end{equs} 
	Combining the estimates, we obtain
	\begin{equ}
		\Big\| \int_{\R^d} \Big(\frac { w^\epsilon_{0,y}(t,x)} { u_\eps (t,x)} - \cev Z (0, y/\epsilon) P_{t} (x-y) \Big) h_0(y) \, dy   \Big\|_p \lesssim  \eps^{\f \kappa 2 -1} t^{\f {2-\kappa} 4} + \eps^{2\mu} t^{-\mu}\;.
	\end{equ}
	Hence, 
	noting that $2\mu < \frac {\kappa } 2 -1$  and 
    using the fact that $t \geq T_0 > 0$, we have
	\begin{equ}[eq:R1eps_bound]
		\Big\|	h_\eps(t,x) - \eps^{1-\f \kappa 2} \log  \big( u_\eps (t,x) \big) - \int_{\R^d} \cev Z (0, y/\epsilon) P_{t} (x-y) h_{0}(y) \, dy \Big\|_p  \lesssim \eps^{2\mu}\;.
	\end{equ}
	
	Let $c = \E [\log \vec Z(t,x)]$, $F(t) = \E [ \log  \big( u (t,x) \big)  ]$ and set $F_\eps(t) = F(t/\eps^2)$. Then, we decompose $h_\epsilon$ as follows:
	\begin{equ}[eq:decomp_heps]
		h_\eps(t,x) - \eps^{1-\f \kappa 2} c =  u_\eps^{\log}(t,x) + \int_{\R^d} P_{t} (x-y) h_{0}(y) \, dy + R_{0, \epsilon}(t) + \sum_{i = 1}^{2} R_{i, \epsilon}(t, x)\;,
	\end{equ}
	where
	\begin{equs}
		u_\eps^{\log}(t,x) &= \eps^{1-\f \kappa 2} ( \log  \big( u_\eps (t,x) \big) - \E [\log  \big( u_\eps (t,x) \big)])\;,\quad R_{0, \epsilon}(t) =  \eps^{1-\f \kappa 2} ( F_\eps(t) - c )\;,\\
		R_{1, \epsilon}(t, x) &= h_\eps(t,x) - \eps^{1-\f \kappa 2} \log  \big( u_\eps (t,x) \big) - \int_{\R^d} \cev Z (0, y/\epsilon) P_{t} (x-y) h_{0}(y) \, dy\;,\\ 
		R_{2, \epsilon} (t,x) &= \int_{\R^d} (\cev Z (0, y/\epsilon) - 1) P_{t} (x-y) h_{0}(y) \, dy \;.
	\end{equs}
	By Theorem~\ref{thm:gen_transf-convergence} with $\fr = \log$, we know that $\nu_{\fr} =  1$ and 
	$u_\eps^{\log}$ converge in probability to $\cU^0$ in $L^p([0, T], \C^\alpha(E))$ for a good coupling of $\xi^\eps$,
	where $\cU^0$ is the solution of the stochastic heat equation \eqref{eq:EW_eqn} with zero initial conditions. Hence, the first two terms in \eqref{eq:decomp_heps} converge in law to the claimed limit $\cU$, the solution of \eqref{eq:AddSHE_h0} with initial conditions $h_0$. It remains to show  that the remainder terms $R_{i, \eps}$ vanish in $L^p([T_0, T], \C^\alpha(E))$.  
	
	Firstly, we show that for $t>0$:
	\begin{equ}
		\abs{F(t) - c}\lesssim 1\wedge |t|^{(2-\kappa)/2}\;.
	\end{equ}
	Then, this implies $|R_{0, \eps}(t)| \lesssim \epsilon^{\f \kappa 2 -1} t^{1 - \f \kappa 2}$, which vanishes uniformly in $t \geq T_0$ as $\eps \to 0$. 
	Recall that 
	\begin{equ}
		F(t)-c = \E\bigl( \log u(t,x)- \log \vec Z(t,x)\bigr)\;.
	\end{equ}
	It follows that, by splitting into the $u<Z$ and $u\ge Z$ case, 
	\begin{equ}
		\Bigl|F(t)-c - \E \Bigl(\frac{u(t,x)- \vec Z(t,x)}{u(t,x)}\Bigr)\Bigr|
		\lesssim \E \Bigl(\frac{u(t,x)- \vec Z(t,x)}{u(t,x) \wedge \vec Z(t,x)}\Bigr)^2\;,
	\end{equ}
	which in turn is bounded by $1\wedge |t|^{(2-\kappa)/2}$ by combining Theorem~\ref{thm:stat_field} and
	Lemma~\ref{lem:moments}\ref{i}.
	The trick now is to write
	\begin{equ}
		\frac{u(t,x)- \vec Z(t,x)}{u(t,x)} = \int \frac{1- \vec Z(0,x)}{u(t,x)}w_{0,y}(t,x)\,dy
	\end{equ}
	and to note that $\vec Z(0,x)$ is independent of both $u(t,x)$ and $w_{0,y}(t,x)$, so that 
	the expectation of this expression vanishes, and concluding the claim on $|F(t)-c |$.
	
	From \eqref{eq:R1eps_bound}, we see that $\|R_{1, \eps} (t,x) \|_p \lesssim \eps ^{2\mu}$ uniformly in $x$ and $t \geq T_0$. Hence, given that the bounds are uniform in space, integrating against any spatial test function $g \in \bar \D$ implies that $R_{0, \eps}$ and 	$R_{1, \eps}$ converge to zero in $L^p([T_0, T], \C^\alpha(E))$. 
	
	It remains to estimate $R_{2, \eps}$. We claim, for any $\delta \in (0, 1)$, uniformly in $\lambda \in (0, 1)$,
	\begin{equ} \label{eq:est_LLN_cevZ}
\biggl\|	\int_{\R^d} R_{2, \eps}(t, x) g_z^\lambda(x) \, dx \biggr\|_2 \lesssim \eps^{\f \kappa 2 - 1} t^{-\f \delta 2} \big( J_{\kappa -2, \delta} (g_z^\lambda )\big)^{\f 12} \;.
	\end{equ}
	By Lemma~\ref{lem:J_kappa_delta_bound}, for any $g \in \bar \D$ we obtain
	\begin{equ}
	 \sup_{t \geq T_0} \biggl\|	\int_{\R^d} R_{2, \eps}(t, x) g_z^\lambda(x) \, dx \biggr\|_2	\lesssim \eps^{\f \kappa 2 - 1} \lambda^{\delta + 1 - \f \kappa 2} \lesssim \eps^{\f \kappa 2 - 1} \lambda^{ 1 - \f \kappa 2} \;.
	\end{equ}
	As in the proof of Lemma~\ref{lem:unif_Calpha_convergence}, this implies that $ \| R_{2, \eps}(t, \cdot) \|_{\C_2^{\alpha_0}} \lesssim \eps^{\f \kappa 2 - 1}$ for any $\alpha_0 \leq  1 - \f \kappa 2$, uniformly in $t \geq T_0$. On the other hand, by the uniform moments bounds, we know
	\begin{equ}
		\sup_{t \geq T_0} \biggl\| \int_{\R^d} R_{2, \eps}(t, x) g_z^\lambda(x) \, dx \biggr\|_p  \lesssim \sup_{t, x} \| R_{2, \eps}(t, x)\|_p \lesssim 1\;.
	\end{equ}
	Then, given $\alpha < 1 - \frac \kappa 2$, picking $p \geq 1$ sufficiently large such that $\alpha < \alpha_0 - \f d p$, by Proposition~\ref{prop:time-holder_distribution_process} and interpolation we conclude that $\sup_{t \geq T_0} \big\| \|  R_{2, \eps}(t,\cdot) \|_{\alpha, E}  \big\|_p \to 0$ as $\eps \to 0$ for $p\in[2,\infty]$. Therefore,  $R_{2, \eps}$ vanishes in $L^p([T_0, T], \C^\alpha(E))$.
	
	It remains to show \eqref{eq:est_LLN_cevZ}, namely that $\cev Z(0, y/\eps )$ satisfies a law of large numbers. We know that $\cev Z$ has constant expectation $1$ and decorrelates in space  by \eqref{eq:g_cov_space_estimate}:  
	\begin{equ}
	 \Lambda(0, y/\epsilon) :=	\cov ( \cev Z (0, 0), \cev Z (0, y/\epsilon) ) \lesssim 1\wedge \eps^{\kappa-2} |y|^{2-\kappa}\;. 
	\end{equ}
	Then, we compute the second moments and use \eqref{eq:Gaussian_est} with $\delta \in (0, 1)$, 
	\begin{equs}
		\E\Big[ \Big( 	\int_{\R^d}&R_{2, \eps}(t, x) g_z^\lambda(x) \, dx\Big)^2 \, \Big] \\
		& = 
		\int_{ \R^{4d}} \Lambda(y_1/\epsilon, y_2/\epsilon) \prod_{i = 1}^2 \big(P_t(x_i - y_i) h_0(y_i) g_z^\lambda(x_i) \big)\, dx\, dy \\
		& \lesssim \eps^{\kappa-2} t^{-\delta} \int_{ \R^{4d}} |y_1 - y_2|^{2-\kappa}  \prod_{i = 1}^2 \big( |x_i - y_i |^{\delta -d} |g_z^\lambda(x_i)| \big)\, dx\, dy\;, 
	\end{equs}
	which implies \eqref{eq:est_LLN_cevZ}, given definition \eqref{eq:defn_J_Gamma0} of $J_{\kappa -2, \delta} $. This concludes the proof of \Cref{thm:main-theorem-KPZ}. 
\end{proof}

\bibliography{references} 
\bibliographystyle{Martin}

\end{document}